\documentclass[a4paper,11pt,oneside]{amsart}
\usepackage{thomas}
\usepackage{xnotes}
\usepackage[foot]{amsaddr}
\usepackage[pagebackref,breaklinks,linktocpage=false,hidelinks,final]{hyperref} 
    \renewcommand*{\backrefalt}[4]{\ifcase #1 (Not cited).\or (Cited p.~#2).\else (Cited pp.~#2).\fi} 
\usepackage{tikz}
\usepackage{float}
\usetikzlibrary{patterns}
\usetikzlibrary{arrows.meta}
\usetikzlibrary{decorations.markings}
\usetikzlibrary{matrix,arrows,calc,fit,cd,positioning,intersections,arrows.meta,braids}
\usepackage{pgfplots}
\usepackage{tikz-3dplot}
\usepackage{mathtools}

\DeclareMathOperator{\spine}{spine}
\makeatletter\def\subsection{\@startsection{subsection}{1}\z@{.7\linespacing\@plus\linespacing}
    {.5\linespacing}{\normalfont\scshape\centering}}\makeatother 

\usepackage[margin=3cm]{geometry}

\title{$\ell^p$ metrics on cell complexes}

\author{Thomas Haettel$^\dagger$}
\address{$^\dagger$IMAG, Univ Montpellier, CNRS, France, IRL 3457, CRM-CNRS, Universit\'{e} de Montr\'{e}al, Canada.}
\thanks{$^\dagger$The first named author was partially supported by l’Agence Nationale de la Recherche (ANR), project ANR-22-CE40-0004.}
\email{thomas.haettel@umontpellier.fr}

\author{Nima Hoda$^\diamond$}
\address{$^\diamond$École normale supérieure, Université
  PSL, CNRS, Paris, France}
\address{$^\diamond$Deptartment of Mathematics, Cornell University,
  Ithaca, NY 14853, USA}
\email{nima.hoda@mail.mcgill.ca}
\thanks{$^\diamond$The second named author was partially supported by the ERC
  grant GroIsRan and an NSERC Postdoctoral Fellowship.}

\author{Harry Petyt$^\star$}
\address{$^\star$University of Oxford, UK}
\email{petyt@maths.ox.ac.uk}

\date{\today}



\begin{document}

\begin{abstract}
Motivated by the observation that groups can be effectively studied using metric spaces modelled on $\ell^1$, $\ell^2$, and $\ell^\infty$ geometry, we consider cell complexes equipped with an $\ell^p$ metric for arbitrary $p$. Under weak conditions that can be checked locally, we establish nonpositive curvature properties of these complexes, such as Busemann-convexity and strong bolicity. We also provide detailed information on the geodesics of these metrics in the special case of CAT(0) cube complexes.
\end{abstract}

\maketitle

\tableofcontents

\section{Introduction}

Three particularly noteworthy classes of metric spaces from recent geometric group theory have been CAT(0) spaces, median metric spaces, and injective metric spaces. Respectively, these can be thought of as being modelled on $\ell^2$, $\ell^1$, and $\ell^\infty$ geometry. As one might expect, the strongest consequences are to be had from the $\ell^1$ model, but correspondingly it is the most restrictive. Conversely, the $\ell^\infty$ model is the most general, but the properties one can obtain are more limited (even though the related Helly graphs enjoy some stronger properties than CAT(0) spaces, leading for instance to biautomaticity~\cite{chalopinchepoigenevoishiraiosajda:helly} and controlled torsion subgroups~\cite{haettelosajda:locally} for groups acting on them). The $\ell^2$ model serves as a happy medium between these two extremes, but it suffers from a different problem: it is extremely difficult to determine whether a given space is CAT(0) without already possessing stronger information, such as the existence of a certain manifold structure or the possibility of using $\ell^1$ methods.

In this article, we consider a natural interpolation between these notions, namely Busemann-convex cell complexes whose cells are given an $\ell^p$ metric. Recall that a metric space $X$ is \emph{Busemann-convex} if for every $x,y\in X$ there is a unique geodesic $\sigma_{xy}:[0,1]\to X$ from $x$ to $y$, and moreover the following holds for all $x,y,x',y'\in X$, $t\in[0,1]$:
\[
\dist(\sigma_{xy}(t),\sigma_{x'y'}(t)) \,\le\, t\dist(y,y')+(1-t)\dist(x,x').
\]
We obtain a number of properties of such complexes, as described in Section~\ref{sec:cell}. 

By relaxing the parameter $p$ to allow finite values larger than 2, we hope to be able to study naturally occurring spaces that either fail to be CAT(0) or are not easily determined to be so. In fact, much of what we do can actually be phrased more generally in terms of spaces with certain bicombings, for which we provide a local criterion (see Theorem~\ref{mthm:local_convexity_condition}). The use of bicombings in nonpositively curved metric spaces has many interesting applications, as illustrated by works such as \cite{lang:injective,descombeslang:convex,descombeslang:flats,kleinerlang:higher,haettel:link,engelwulff:coronas,basso:extending,karlsson:metric}.

One important geometric property of cell complexes with the $\ell^2$ metric that fails for the $\ell^\infty$ metric is that of strong bolicity (see Definition~\ref{def:strong_bolicity}, Remark~\ref{rem:bolicity}). It is a generalisation to general metric spaces of notions from the theory of Banach spaces, and was introduced by Kasparov and Skandalis in their work on the Novikov conjecture in~\cite{kasparovskandalis:groupes}. It was then used by Lafforgue for the Baum--Connes conjecture in~\cite{lafforgue:theorie}, which led to a proof for hyperbolic groups by Mineyev and Yu in~\cite{mineyevyu:baumconnes}. We show that, under mild conditions, a cell complex with the $\ell^p$ metric is strongly bolic when $p$ is finite (see Theorem~\ref{mthm:criterion_piecewise_lp_bolicity}).

A motivating idea here comes from hyperbolic groups. In the $\ell^2$ setting it is a well-known open problem whether every hyperbolic group admits a proper cocompact action on a CAT(0) space. On the other hand, if we consider $p=\infty$ then a result of Lang \cite{lang:injective} tells us that every hyperbolic group has such an action on an injective space with a natural cell structure. The $\ell^p$ framework suggests a natural weakening of the CAT(0) question: is it the case that for each hyperbolic group $G$ there is a value of $p$ such that $G$ acts properly cocompactly on a Busemann-convex cell complex whose cells have an $\ell^p$ metric?  Positive evidence for this comes from Yu's result that every hyperbolic group admits a proper affine action on an $\ell^p$-space \cite{yu:hyperbolic} (see also \cite{alvarezlafforgue:actions}), which holds in spite of the fact that (finitely generated) groups with Property (T) cannot have unbounded orbits on $\ell^2$ \cite{delorme:cohomologie,guichardet:surcohomologie:2}.

As a nice source of examples, we work out in detail the case of a CAT(0) cube complex endowed with the $\ell^p$ metric, following work of Ardila, Owen, and Sullivant for the CAT(0) metric (see~\cite{ardilaowensullivant:geodesics}). We prove that, with the $\ell^p$ metric, CAT(0) cube complexes are Busemann-convex and strongly bolic (Theorem~\ref{mthm:busemann_cube_complex}). We also give an explicit local characterisation of geodesics (Theorem~\ref{mthm:description_local_geodesics}) and derive a local distance formula (Theorem~\ref{mthm:CCC_distance_formula}).

\subsection{Cell complexes} \label{sec:cell}

Consider a finite-dimensional normed cell complex, i.e.\ a cell complex whose cells are convex polyhedra in normed vector spaces, glued by isometries of faces. We establish a simple criterion ensuring that a given continuous bicombing is convex. Roughly, a \emph{bicombing} is a choice of path between each pair of points; see Definition~\ref{def:bicombing}. Note that this result also applies to norms that are not uniquely geodesic.

\bmthm[Theorem~\ref{thm:criterion_busemann_convex}] \label{mthm:local_convexity_condition}
Let $X$ be a piecewise normed cell complex with finitely many shapes. Assume that the following conditions hold.
  \begin{enumerate}
  \item $X$ is simply connected.
  \item $X$ locally admits a consistent geodesic bicombing.
  \item For any two intersecting maximal cells $A,B$ of $X$, the union $A \cup B$, with the induced length metric, is Busemann-convex.
  \end{enumerate}
Then $X$ is Busemann-convex (in particular, it is uniquely geodesic).
\emthm

See Theorem~\ref{thm:gluing_busemann_convex} for a stronger statement,
with more precise assumptions. This result can be used to provide new examples of spaces with unique convex geodesic bicombings, called CUB spaces in~\cite{haettel:link}.

\mk

Now consider a finite-dimensional normed complex. We establish a criterion ensuring that the space is strongly bolic, see Theorem~\ref{thm:convex_bicombing_implies_bolicity} for a more precise version.

\bmthm[Theorem~\ref{thm:criterion_piecewise_lp_bolicity}] \label{mthm:criterion_piecewise_lp_bolicity}
Let $p \in [2,\infty)$, and let $X$ be a piecewise $\ell^p$ cell complex with finitely many shapes satisfying the following:
\bit
\item $X$ is simply connected.
\item $X$ locally admits a consistent geodesic bicombing.
\item For any two intersecting maximal cells $A,B$ of $X$, the union $A \cup B$, with the induced length metric, is Busemann-convex.
\eit
Then $X$ is Busemann-convex and strongly bolic.
\emthm

Our strategy for proving Theorem~\ref{mthm:criterion_piecewise_lp_bolicity} involves proving that such complexes are \emph{uniformly convex} and \emph{uniformly smooth}. These concepts play an important role in analysis and Banach-space theory, having been introduced in \cite{clarkson:uniformly}. They were shown to be dual in \cite{day:uniform}. They come with associated constants, and there has been a great deal of work around optimising these---see \cite{ballcarlenlieb:sharp} for more discussion. Both conditions were abstracted in an $\ell^2$ form to the nonlinear metric setting by Ohta \cite{ohta:uniform} (following \cite{bucherkarlsson:ondefinition}). The conditions we consider here are $\ell^p$ modifications of Ohta's definitions; see Definitions~\ref{def:uniform_convex} and~\ref{def:uniform_smooth}.

Key motivation for strong bolicity comes from Lafforgue's work on the Baum--Connes conjecture \cite{lafforgue:theorie}. To follow this strategy, one needs to find a proper action on a strongly bolic space. We observe that Bridson's splitting result for actions on CAT(0) spaces extends to strongly bolic metric spaces, see Theorem~\ref{thm:centralisers_split}. As a consequence, this strategy cannot be applied to mapping class groups of surfaces.

\bmthm[Corollaries~\ref{cor:mcg_not_lp} and~\ref{cor:mcg_not_bolic}] \label{mthm:mcg_not_bolic}
For $g \geq 3$, the mapping class group $\Mod(S_g)$ has no action by isometries on a strongly bolic metric space inducing a quasi-isometric embedding.  In particular, there is no proper, cobounded action of $\Mod(S_g)$ on a strongly bolic metric space.
\emthm

More precisely, we show that, for every action of the mapping class group on a strongly B1 space (see Definition~\ref{def:strong_bolicity}), the orbit of each Dehn twist is distorted.

\subsection{CAT(0) cube complexes}

To complement the results described above, we consider the concrete case of CAT(0) cube complexes. We prove that any CAT(0) cube complex, when endowed with the $\ell^p$ metric, is uniquely geodesic, and even Busemann-convex.

\bmthm[Theorem~\ref{thm:busemann_cube_complex}] \label{mthm:busemann_cube_complex}
Let $X$ denote any CAT(0) cube complex, endowed with the standard piecewise $\ell^p$-metric, for some $p \in (1,\infty)$. Then $X$ is Busemann-convex and uniformly convex. If $p \geq 2$, then $X$ is also uniformly smooth and strongly bolic.
\emthm

We use this result to deduce that for the limiting cases $p=1$ and $p=\infty$, even though the metric is not uniquely geodesic, there still exists a unique convex bicombing.

\bmthm[Theorem~\ref{thm:cube_complex_unique_bicombing_allp}] \label{mthm:cube_complex_unique_bicombing_allp}
Let $X$ denote any CAT(0) cube complex, endowed with the standard piecewise $\ell^p$-metric, for some $p \in [1,\infty]$. Then $X$ is CUB, i.e. it admits a unique convex geodesic bicombing $\sigma^p$, which varies continuously in $p$.
\emthm

Note that, to our knowledge, the existence a unique convex bicombing is also new for the piecewise $\ell^1$ metric on a CAT(0) cube complex. It is perhaps surprising that this is established using values of $p$ greater than 1.

In order to prove this result, we describe very precisely the behaviour of local geodesics via two simple conditions. This generalises results of Ardila--Owen--Sullivant for the CAT(0) metric \cite{ardilaowensullivant:geodesics}, about which there has been a significant amount of interest in computational geometry.

\bmthm[Theorem~\ref{thm:description_local_geodesics}] \label{mthm:description_local_geodesics}
Let $X$ be a CAT(0) cube complex, endowed with the piecewise $\ell^p$-metric for some $p\in(1,\infty)$. A piecewise affine path is a local geodesic in $X$ if and only if it satisfies the \emph{zero-tension} condition and the \emph{no-shortcut} condition.
\emthm

The zero-tension and no-shortcut conditions are discussed in Sections~\ref{subsec:zero_tension} and~\ref{subsec:no_shortcut}, respectively. Heuristically, the zero-tension condition says that the path does not change velocity when it changes cube, and the no-shortcut condition says that it doesn't go the long way round a vertex.

We are then able to describe the local structure of the cube complex, and give a simple formula for the local distances.

\bmthm[Distance formula, Proposition~\ref{prop:unique_decomposition} and Lemma~\ref{lem:distance_formula}] \label{mthm:CCC_distance_formula}
Let $X$ be a CAT(0) cube complex with the $\ell^p$-metric, for some $p\in(1,\infty)$. Let $C$ and $D$ be cubes in $X$ whose intersection is a vertex $v$. Let $x \in C$ and $y \in D$ be interior points. There exist unique maximal decompositions $C=\prod_{j=1}^kA_j$ and $D=\prod_{j=1}^kB_j$ such that all cubes of the form $B_1 \times \dots \times B_j \times A_{j+1} \times \dots \times A_k$ belong to $X$ and  
\begin{align*} 
\f{\|x-v\|_{A_1}}{\|y-v\|_{B_1}} \,<\, \f{\|x-v\|_{A_2}}{\|y-v\|_{B_2}} 
    \,<\, \dots \,<\, \f{\|x-v\|_{A_k}}{\|y-v\|_{B_k}}, 
\end{align*}
where, for instance, $\|x-v\|_{A_1}$ denotes the $\ell^p$ distance from $v$ to the orthogonal projection of $x$ to $A_1$.
In addition, the $\ell^p$-distance from $x$ to $y$ is given by
\[
\left\|\big(\|x-v\|_{A_1},\dots,\|x-v\|_{A_k}\big)
    \,+\, \big(\|y-v\|_{B_1},\dots,\|y-v\|_{B_k}\big)\right\|.
\]
\emthm

Though $C$ and $D$ meet at a point, there may be cubes ``in the corner between them'', as in Figure~\ref{fig:distance_formula}. We can think of these as sitting inside a $2\times2$ product coming from $C$ and $D$. The decomposition in Theorem~\ref{mthm:CCC_distance_formula} then represents the biggest such product in which $\{x,y,v\}$ isometrically embeds, and the distance formula then simply measures the distance in that product.

\subsection{Organisation of the article}

In Section~\ref{sec:bicombing}, we review definitions of bicombings, and we prove Theorem~\ref{mthm:local_convexity_condition} on Busemann-convexity of piecewise normed cell complexes.

In Section~\ref{sec:examples}, we present several basic examples of cell complexes with $\ell^p$ metrics, exhibiting various behaviours regarding Busemann-convexity when the value of $p$ varies.

Section~\ref{sec:bolicity} concerns strong bolicity, uniform smoothness, and uniform convexity. We prove Theorem~\ref{mthm:criterion_piecewise_lp_bolicity}, showing how strong bolicity can be established using a convex bicombing. We also prove Theorem~\ref{mthm:mcg_not_bolic}.

In Section~\ref{sec:lp_cube_complex}, we consider CAT(0) cube complexes with the $\ell^p$ metric, proving Theorems~\ref{mthm:busemann_cube_complex}, \ref{mthm:cube_complex_unique_bicombing_allp}, \ref{mthm:description_local_geodesics}, and~\ref{mthm:CCC_distance_formula}

\subsection{Acknowledgments}

The authors warmly thank Indira Chatterji, Urs Lang, and Constantin Vernicos for interesting discussions. We also thank the referee for useful comments that benefited the writing of the paper.

\section{Busemann-convexity of normed polyhedral complexes} \label{sec:bicombing}

We are interested in describing a simple local criterion ensuring that a given continuous bicombing on a normed polyhedral complex is convex. A related result is~\cite[Theorem~1.2]{buragoivanov:polyhedral}, which is concerned with locally uniquely geodesic spaces, and remarks that the question of Busemann-convexity is subtle.

\begin{defi}[Bicombing] \label{def:bicombing}
A \emph{bicombing} on a metric space $X$ is a continuous map $\sigma:X \times X \times [0,1] \ra X$, such that for each $x,y \in X$, one has $\sigma_{xy}(0)=x$ and $\sigma_{xy}(1)=y$. It is called:
\bit
\item \emph{geodesic} if for each $x,y \in X$, the map $t \mapsto \sigma_{xy}(t)$ is a constant speed (reparametrised) geodesic from $x$ to $y$;
\item \emph{consistent} if subpaths of bicombing paths are bicombing paths: for each $x,y \in X$,
    $s \leq t$, and $u \in [0,1]$, we have
    $\sigma_{\sigma_{xy}(s)\sigma_{xy}(t)}(u) = \sigma_{xy}\bigl((1-u)s+ut\bigr)$;
\item \emph{convex} if for each $x,y,x',y' \in X$, the function $t \mapsto d(\sigma_{xy}(t),\sigma_{x'y'}(t))$ is convex.
\eit
A \emph{local consistent bicombing} on $X$ is a collection of consistent bicombings $\sigma_U:U \times U \times [0,1] \ra X$, where $U$ varies over some open cover ${\mathcal U}$ of $X$, such that for every intersecting $U,V \in {\mathcal U}$ we have $\sigma_U|_{(U \cap V)^2}=\sigma_V|_{(U \cap V)^2}$.
\end{defi}
Thus, a metric space is Busemann-convex if it is uniquely geodesic and the unique geodesic bicombing is convex.

It turns out that we can check convexity of a consistent bicombing just by looking at midpoints.

\blem \label{lem:busemann_sufficient}
Let $X$ denote a metric space with a consistent bicombing $\sigma$. If
\[
d(\sigma_{xy}(\f12),\sigma_{xy'}(\f12)) \leq \f{d(y,y')}{2}
\]
holds for all $x,y,y'\in X$, then $\sigma$ is a convex bicombing.
\elem

\begin{proof}
Given $x,x',y,y'\in X$, we know that $d(\sigma_{xy}(\f12),\sigma_{x'y'}(\f12))\le\f12d(y,y')+\f12d(x,x')$. By iterating the assumption, $d(\sigma_{xy}(\f14),\sigma_{x'y'}(\f14))\le\f34d(x,x')+\f14d(y,y')$, and similarly at time $\f34$. Repeating, we get that for every dyadic $\f a{2^b}\in[0,1]$, we have $d(\sigma_{xy}(\f a{2^b}),\sigma_{x'y'}(\f a{2^b})) \le (1-\f a{2^b})d(x,x')+\f a{2^b}d(y,y')$. The dyadics are dense in $[0,1]$ and the function $t\mapsto d(\sigma_{xy}(t),\sigma_{x'y'}(t))$ is continuous, so it follows that $d(\sigma_{xy}(t),\sigma_{x'y'}(t))\le(1-t)d(x,x')+td(y,y')$ for all $t\in[0,1]$. This, along with consistency, implies that the bicombing $\sigma$ is convex.
\end{proof}

The convexity property in Busemann-convexity needs only to be checked locally.

\bthm[\cite{alexanderbishop:hadamard},\cite{miesch:cartan}] \label{thm:cartan_hadamard}
Let $X$ be a complete, simply connected, geodesic metric space. If every point of $X$ has a Busemann-convex neighbourhood, then $X$ is Busemann-convex. More generally, if $X$ has a local consistent geodesic bicombing $\sigma$ that is convex, then $X$ has a unique convex, consistent, geodesic bicombing that restricts to $\sigma$.
\ethm

We will now define cone complexes, which will be useful to describe the geometry of the neighbourhood of a point in a cell complex.

\begin{defi}[Cone polyhedron]
A \emph{cone polyhedron} $P$ is a finite intersection
$\bigcap_{i=1}^k H_i$ of linear half-spaces of some
finite-dimensional normed space $\bigl(\R^n, |\cdot|\bigr)$ such that
$0 \in \partial H_i$ for all $i$.  A \emph{face} or \emph{cell} of $P$
is $P \cap \bigcap_{i \in I} \partial H_i$ for some
$I \subseteq \{1,2, \ldots, k\}$.  The \emph{spine} of $P$ is the face
$\bigcap_{i=1}^k \partial H_i$, it is also the largest linear
subspace of $P$.  The spine of any face of $P$
coincides with the spine of $P$.
\end{defi}

\begin{defi}[Cone complex]
Let $V$ be a normed vector space and let $\{P_i\}$ be a collection of cone polyhedra with $V<\spine P_i$. The length metric space obtained by gluing the $P_i$ along identifications of $V$ is called a \emph{cone complex}. Its spine is $V$.
\end{defi}

Fix $x_0 \in X$, let $C$ be the minimal closed cell containing $x_0$
and let $X_C$ be the closed star of $C$ (the union of all closed cells containing $C$).  Given $x \in X_C$ there is a
unique affine map $\phi_x$ from $[0,1]$ to the minimal closed cell of
$X$ that contains both $x_0$ and $x$ such that $\phi_x(0) = x_0$ and
$\phi_x(1) = x$. Let $rx$ denote the image $\phi_x(r)$.  For
$n \in \N$, let $nX_C$ be $X_C$ with the metric scaled up by $n$.
Then for $n < m$, there is an isometric embedding $nX_C \to mX_C$
given by $x \mapsto \frac{n}{m}x$.  Viewing this isometric embedding
as an inclusion, we define the \emph{tangent cone} $T_{x_0} X$ of $X$
at $x_0$ as the ascending union $\bigcup_n nX_C$.

We say that a cone complex $X$ is \emph{locally conical} if, for every point $x \in X$, the open star of the minimal closed cell containing $x$ is a neighbourhood of $x$ in $X$. (This is a different sense of the word ``conical'' compared to that of a conical bicombing.) For instance, if $X$ has finitely many shapes, then $X$ is locally conical (recall that a cell complex is said to have finitely many shapes if it has finitely many isometry types of cells). The following is immediate.

\blem \label{lem:cone_complex_tangent_cone}
Let $X$ denote a locally conical cone complex. Given $x \in X$, let $C_x$ be the minimal closed cell containing $x$. The metric tangent cone $T_x X$ of $X$ at $x$ is a locally conical cone complex such that any small enough ball centred at $x$ in $X$ is isometric to a ball centred at $x$ in $T_x X$. Moreover, the spine of $T_x X$ has the same dimension as $C_x$.
\elem

If $X$ is a locally conical cone complex and $\sigma$ is a bicombing on $X$, we say that $\sigma$ is \emph{locally dilation-invariant} if for any $c \in X$, there exists $\eps>0$ such that, for any $x,y \in B(c,\eps)$, and for any $\lambda \in [0,1]$, we have $\sigma(c+\lambda x,c+\lambda y)=c+\lambda \sigma(x,y)$ with respect to the piecewise affine structure of $X$. More precisely, for any $c \in X$, we know by local conicality and by Lemma~\ref{lem:cone_complex_tangent_cone} that any small enough ball centred at $c$ in $X$ is isometric to a ball centred at $c$ in $T_cX$. Moreover, $T_cX$ is a cone complex and $c \in \spine(T_cX)$. Hence homotheties centred at $c$ are well-defined in $T_cX$, and this explains the affine notations $c+\lambda x$, $c+\lambda y$ and $c+\lambda \sigma(x,y)$.

\begin{thm} \label{thm:gluing_busemann_convex}
Let $X$ be a locally conical, finite-dimensional cone complex, and suppose that $\sigma$ is a consistent (not \emph{a priori} geodesic) bicombing that restricts to the constant-speed affine bicombing on each cell. Assume that the following conditions hold.
\begin{enumerate}
\item   $\sigma$ is locally dilation-invariant.
\item   For any two cells $C,C'$ of $X$, and for any $x \in C$ and $y,z \in C'$ such that $\sigma_{x,y}$ and $\sigma_{x,z}$ are contained in $C \cup C'$, the map $t \mapsto d(\sigma_{x,y}(t),\sigma_{x,z}(t))$ is convex.
\end{enumerate}
Then $\sigma$ is a convex geodesic bicombing.
\end{thm}

\begin{proof}
  We proceed by induction on the codimension of the spine of $X$.
  
  If the spine of $X$ has codimension $0$, then $X$ is $\R^n$ endowed
  with some norm.  So by assumption, $\sigma$ is the constant speed
  affine bicombing, which is convex.
  
  Assume now that the spine of $X$ has codimension at least $1$. Since
  $\sigma$ is consistent, according to Lemma~\ref{lem:busemann_sufficient} it suffices to show that
  \begin{equation}
    \label{eqn:convineq}
    d\bigl(\sigma_{x,y}(\tfrac{1}{2}),\sigma_{x,z}(\tfrac{1}{2})\bigr) \leq
    \tfrac{1}{2}d(y,z)
  \end{equation}
  for any $x,y,z \in X$ in order to establish the convexity of
  $\sigma$.
  
  For $s \in [0,1]$, let $\bar s = \sigma_{y,z}(s)$.  Consider the set
  $S_{\spine}$ of all $s \in [0,1]$ for which $\sigma_{x,\bar s}$
  intersects $\spine(X)$.  Since $\{\sigma_{x,y}\}_{x,y}$ is
  continuous, if $S_{\spine} \neq \varnothing$ then $S_{\spine}$ has a
  minimum and a maximum.  In this case, set
  $s_{\min} = \min S_{\spine}$ and $s_{\max} = \max S_{\spine}$;
  otherwise set $s_{\min} = 1$ and $s_{\max} = 0$, so that $s_{\min}>s_{\max}$. See Figure~\ref{fig:spine}.

  We will prove the inequalities
  \begin{align}
  \label{eqn:yminconv} d\bigl(\sigma_{x,y}(\tfrac{1}{2}),\sigma_{x,\bar s_{\min}}(\tfrac{1}{2})\bigr) &\leq \tfrac{1}{2}d(y,\bar s_{\min}) \\
  \label{eqn:zmaxconv} d\bigl(\sigma_{x,\bar s_{\max}}(\tfrac{1}{2}),\sigma_{x,z}(\tfrac{1}{2})\bigr) &\leq
  \tfrac{1}{2}d(\bar s_{\max},z)
  \end{align}
  and, in the case where $s_{\min} \leq s_{\max}$, the inequality
  \begin{equation}
    \label{eqn:minmaxconv}
    d\bigl(\sigma_{x,\bar s_{\min}}(\tfrac{1}{2}),\sigma_{x,\bar s_{\max}}(\tfrac{1}{2})\bigr) \leq \tfrac{1}{2}d(\bar s_{\min},\bar s_{\max}).
  \end{equation}
  Together, these inequalities will imply \eqref{eqn:convineq}.
  
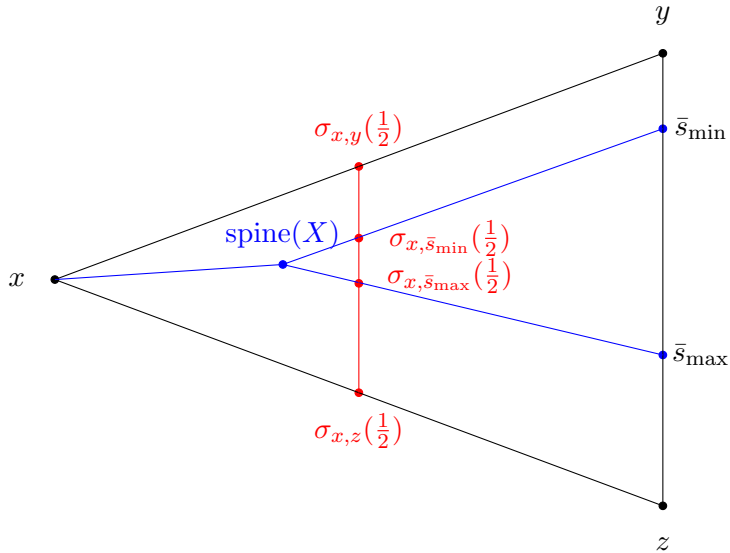
\begin{figure}[ht]
\centering
\begin{tikzpicture}
\def \p {0.05}
\def \op {0.3}
\def \gris {blue}
\draw[fill] (0,0) circle (\p) node(x) {};
\draw[fill] (8,3) circle (\p) node(y) {};
\draw[fill] (8,-3) circle (\p) node(z) {};
\draw[fill,blue] (3,0.2) circle (\p) node(s) {};
\draw[fill,blue] (8,2) circle (\p) node(smin) {};
\draw[fill,blue] (8,-1) circle (\p) node(smax) {};
\draw[fill,red] (4,1.5) circle (\p) node(my) {};
\draw[fill,red] (4,0.55) circle (\p) node(msmin) {};
\draw[fill,red] (4,-0.05) circle (\p) node(msmax) {};
\draw[fill,red] (4,-1.5) circle (\p) node(mz) {};

\draw[black] (z.center) -- (x.center) -- (y.center) -- (z.center);
\draw[blue] (x.center) -- (s.center) -- (smin.center);
\draw[blue] (s.center) -- (smax.center);
\draw[red] (my.center) -- (mz.center);

\node at ([xshift=-0.5cm]x) {$x$};
\node at ([yshift=0.5cm]y) {$y$};
\node at ([yshift=-0.5cm]z) {$z$};
\node at ([yshift=0.5cm]my) {${\color{red} \sigma_{x,y}(\tfrac{1}{2})}$};
\node at ([yshift=-0.5cm]mz) {${\color{red} \sigma_{x,z}(\tfrac{1}{2})}$};
\node at ([xshift=1.2cm]msmin) {${\color{red} \sigma_{x,\bar s_{\min}}(\tfrac{1}{2})}$};
\node at ([xshift=1.2cm,yshift=0.1cm]msmax) {${\color{red} \sigma_{x,\bar s_{\max}}(\tfrac{1}{2})}$};
\node at ([yshift=0.4cm]s) {${\color{blue} \spine(X)}$};
\node at ([xshift=0.5cm]smin) {$\bar s_{\min}$};
\node at ([xshift=0.5cm]smax) {$\bar s_{\max}$};

\end{tikzpicture}
\caption{Proof of Theorem~\ref{thm:gluing_busemann_convex}}
\label{fig:spine}
\end{figure}
  
  We begin with the proof of \eqref{eqn:yminconv}. (The proof of \eqref{eqn:zmaxconv} is identical.) We may assume that
  $s_{\min} > 0$ since otherwise we have $\bar s_{\min} = y$ so that
  $d\bigl(\sigma_{x,y}(\tfrac{1}{2}),\sigma_{x,\bar
    s_{\min}}(\tfrac{1}{2})\bigr) = 0$.
  
  For $s\in [0,s_{\min})$, $t\in[0,1]$, let $C_{s,t}$ be the
  minimal closed cell containing $\sigma_{x,\bar s}(t)$. Since $X$ is
  locally conical, according to
  Lemma~\ref{lem:cone_complex_tangent_cone}, we can consider the
  tangent cone complex $T_{s,t}=T_{\sigma_{x,\bar s}(t)} X$. We will
  see that it satisfies all assumptions of the theorem.
  
  Since the bicombing $\sigma$ is locally dilation-invariant, it induces a consistent, geodesic bicombing $\sigma^{s,t}$ on $T_{s,t}$ that is locally dilation-invariant. Furthermore, $\sigma^{st}$ restricts to the constant speed affine bicombing on each cell because $\sigma$ does. Assume that there are cells $C,C'$ of $T_{st}$ and $x \in C$, $y,z \in C'$ such that $\sigma^{st}_{x,y},\sigma^{st}_{x,z} \subset C \cup C'$. Since $T_{st}$ is the tangent cone of $X$ at $\sigma_{x,\bar s}(t)$, we see that the convexity property for $\sigma$ implies that the function $u \mapsto d(\sigma^{st}_{x,y}(u),\sigma^{st}_{x,z}(u))$ is convex.
  
  Moreover, the dimension of $T_{st}$ is at most the dimension of $X$, and the dimension of the spine of $T_{st}$ is greater than the dimension of the spine of $X$. Hence the codimension of the spine of $T_{st}$ is less than the codimension of the spine of $X$. By induction, we deduce that $\sigma^{st}$ is a convex bicombing.
  
  Since $T_{st}$ is obtained as a very simple scaling of $X$, we deduce that $\sigma$ restricts to a convex bicombing on the open star of $C_{st}$. By continuity of $\sigma$ there exists $\delta_{s,t} > 0$ such that the image of $[s-\delta_{s,t}, s+\delta_{s,t}] \times [t - \delta_{s,t}, t+\delta{s,t}]$ under $(s,t) \mapsto \sigma_{x,\bar s}(t)$ is contained in the open star of $C_{s,t}$. By compactness of $[0,1]$, we deduce that for each $s \in [0,s_{\min})$, there exists $\eps_s>0$ such that, for each $s' \in [0,s_{\min})$ with $|s-s'|<\eps_s$, the function $t \in [0,1] \mapsto d(\sigma_{x,\bar s}(t),\sigma_{x,\bar s'}(t))$ is convex. As $\sigma$ is continuous, it follows that $t \mapsto d(\sigma_{x,y}(t),\sigma_{x,\bar s_{\min}}(t))$ is convex, and in particular $d\bigl(\sigma_{x,y}(\tfrac{1}{2}),\sigma_{x,\bar s_{\min}}(\tfrac{1}{2})\bigr) \leq \tfrac{1}{2}d(y,\bar s_{\min})$.

If $S_{\spine}=\varnothing$ then we are done, so it remains to establish \eqref{eqn:minmaxconv} under the assumption that $s_{\min} < s_{\max}$. According to Lemma~\ref{lem:finite_number_cells} below, the combing line
  $\sigma_{\bar s_{\min},\bar s_{\max}}$ is contained in a finite
  union of cells. Thus, by consistency we may (after subdividing) assume that
  $\bar s_{\min}$ and $\bar s_{\max}$ are contained in a common cell
  $C'$. Let $C$ denote the minimal closed cell containing $x$. Both combing lines $\sigma_{x,\bar s_{\min}}$ and
  $\sigma_{x,\bar s_{\max}}$ are contained in $C \cup C'$. By
  assumption, we deduce that
  $d(\sigma_{x,\bar s_{\min}}(\f{1}{2}),\sigma_{x,\bar
    s_{\max}}(\f{1}{2})) \leq \f{1}{2}d(\bar s_{\min},\bar s_{\max})$. This shows that $\sigma$ is convex.
    
Now observe that a convex bicombing is necessarily geodesic. Indeed, for any $x,y \in X$, since $\sigma_{x,x}=\{x\}$ and $\sigma_{y,y}=\{y\}$, we know that the functions $t\mapsto d(\sigma_{x,y}(t),x)$ and $t \mapsto d(\sigma_{x,y}(t),y)$ are convex. This means that the function $f:t \mapsto d(\sigma(x,y,t),x)+d(\sigma(x,y,t),y)$ is convex, bounded below by $d(x,y)$, and equal to $d(x,y)$ at $t=0$ and $t=1$. Therefore $f=d(x,y)$. That implies that $\sigma_{x,y}$ is a constant speed geodesic from $x$ to $y$.
\end{proof}

\blem \label{lem:finite_number_cells}
Consider $X$ as in Theorem~\ref{thm:gluing_busemann_convex}. For any points $y,y' \in X$, the combing line $\sigma_{yy'}$ is contained in a union of finitely many cells of $X$.
\elem

\bp
We proceed by induction on the codimension of $\spine X$. 
If $\spine X$ has codimension $0$, then $X$ is affine and the result is clear, so assume that $\spine X$ has codimension at least $1$. 

Assume first that $\sigma_{yy'}$ intersects the spine of $X$. Let $z \in \sigma_{yy'} \cap \spine(X)$. By the assumptions on $\sigma$, we know that $\sigma_{yy'}$ is the union of the two affine segments $\sigma_{yy'}=[y,z] \cup [z,y']$, which is contained in the union of two cells.

Assume now that $\sigma_{yy'}$ does not intersect the spine of $X$. Fix $t \in [0,1]$. Since $X$ is locally conical, there exists $r_t>0$ such that the ball $B_X(\sigma_{yy'}(t),r_t)$ is contained in the open star of the minimal cell $C_t$ containing $\sigma_{yy'}(t)$. By continuity of $\sigma$, there exists $\delta_t>0$ such that, for any $t \in [0,1]$ with $|t-t'| < \delta_t$, the point $\sigma_{yy'}(t')$ also lies in the open star of the minimal cell of $C_t$. Let $T_t$ denote the tangent cone to $X$ at $\sigma_{yy'}(t)$. Since the bicombing $\sigma_t$ on $T_t$ is the scaling of $\sigma$, and since the codimension of the spine of $T_t$ is less than the codimension of the spine of $X$, we deduce by induction that there the image of $[0,1] \cap (t-\delta_t,t+\delta_t)$ under $\sigma_{yy'}$ is contained in the union of finitely many cells of $X$. Since the segment $[0,1]$ is compact, we conclude that $\sigma_{yy'}$ is contained in the union of finitely many cells of $X$.
\ep

\brk \label{rk:gluing_busemann_convex}
In fact, the proof of Theorem~\ref{thm:gluing_busemann_convex} provides the following more precise statement.

Let $X$ be a locally conical, finite-dimensional cone complex, and let $U \subset V \subset X$ denote non-empty open subsets of $X$. Assume that the following conditions hold.
  \begin{enumerate}
  \item There exists a bicombing $\sigma : V \times V \times [0,1] \ra X$, with $\sigma(U \times U) \subset V$, that is partially consistent, i.e. for each $x,y \in V$, $s \leq t$, and $u \in [0,1]$ such that $\sigma_{xy}(s),\sigma_{xy}(t) \in V$, we have $\sigma_{\sigma_{xy}(s)\sigma_{xy}(t)}(u) = \sigma_{xy}\bigl((1-u)s+ut\bigr)$.
  \item $\sigma$ is locally dilation-invariant.
  \item For each cell $C$ of $X$, the restriction of $\sigma$ to $(C \cap V) \times (C \cap V)$ is the constant speed affine bicombing.
  \item For any two cells $C,C'$ of $X$, and for any $x \in C \cap U$ and $y,z \in C' \cap U$ such that $\sigma_{xy}$ and $\sigma_{xz}$ are contained in $C \cup C'$, the map $t \mapsto d(\sigma_{xy}(t),\sigma_{xz}(t))$ is convex.
  \end{enumerate}
  Then $\sigma$ is a convex geodesic bicombing on $V$.
\erk

\brk
Note that the assumption that the union of two cones $C \cup C'$ is Busemann-convex may not be removed from Theorem~\ref{thm:gluing_busemann_convex}. Already there is a counterexample in~\cite{piateksamulewicz:gluing} with a union of $3$ half-planes. Here is a simpler example with a gluing of two half-spaces.

Consider $H_1=\{(x,y) \in \R^2, x \leq 0\}$, with the $\ell^2$ metric, and $H_2=\{(x,y) \in \R^2, x \geq 0\}$, with the $\ell^3$ metric. Let $X$ be the cone complex obtained by gluing $H_1$ and $H_2$ along the line $\{x=0\}$, endowed with the induced length metric. It is not hard to see that $X$ is uniquely geodesic.

However, consider the points $x=(-1,0) \in H_1$, $y=(1,0) \in H_2$, and $z=(0,\sqrt{3})$. We have $d(x,y)=2=d(x,z)$ and $d(y,z)=(1+\sqrt{3}^3)^{\frac{1}{3}}$. The midpoint of $x$ and $y$ is $m(x,y)=(0,0)$, and the midpoint between $x$ and $z$ is $\left(\f{-1}{2},\f{\sqrt{3}}{2}\right)$. So the distance between midpoints is $d(m(x,y),m(x,z)) = 1 > \frac12(1+\sqrt{3}^3)^{\frac{1}{3}} = \frac12d(y,z)$. Hence $X$ is not Busemann-convex.
\erk

We now turn to the stronger property of Busemann-convexity. The following result shows how local convexity can be upgraded to local Busemann-convexity.

\bpro \label{pro:locally_convex_busemann}
Let $X$ be a locally conical piecewise normed cell complex that admits a local convex, consistent, geodesic bicombing. Suppose that for any two intersecting maximal cells $A$ and $B$ of $X$, the union $A \cup B$, with the induced length metric, is Busemann-convex. Then $X$ is locally Busemann-convex.
\epro

\bp
Consider an open subset $U \subset X$ with a convex, consistent, geodesic bicombing $\sigma$. Fix $x,y \in U$, and consider any constant speed geodesic $\gamma$ from $x$ to $y$. By assumption about intersections of maximal cells, we know that $\gamma$ locally coincides with $\sigma$. We deduce that the distance between $\gamma$ and $\sigma_{x,y}$ is a convex function. Hence $\gamma=\sigma_{x,y}$, so $U$ is uniquely geodesic. We deduce from convexity of $\sigma$ that $U$ is Busemann-convex.
\ep

We may now combine Theorem~\ref{thm:gluing_busemann_convex} and Proposition~\ref{pro:locally_convex_busemann} to state a simple, applicable result ensuring that a given space is Busemann-convex. 

\begin{thm} \label{thm:criterion_busemann_convex}
Let $X$ be a piecewise normed cell complex with finitely many shapes. If the following conditions hold, then $X$ is Busemann-convex.
\begin{enumerate}
\item $X$ is simply connected.
\item $X$ admits a local consistent geodesic bicombing.
\item For any two intersecting maximal cells $A,B$ of $X$, the union $A \cup B$, with the induced length metric, is Busemann-convex.
\end{enumerate}
\end{thm}

\bp
We will see that these conditions imply the conditions of Theorem~\ref{thm:gluing_busemann_convex} locally. Let us fix a point $x_0 \in X$, and consider $r>0$ such that the open ball $V=B(x_0,3r)$ is contained in the open star of the minimal closed face containing $x_0$. Let $U=B(x_0,r)\subset V$. By assumption, there exists a continuous, consistent geodesic bicombing $\sigma:V \times V \ra X$ such that $\sigma(U \times U) \subset V$. We will apply Remark~\ref{rk:gluing_busemann_convex} to show that $\sigma$ is convex.

First note that, since $X$ has finitely many shapes, the neighbourhood $V$ of $x_0$ is isometric to a neighbourhood in a locally conical cone complex $Y$.

We will prove that $\sigma$ is locally dilation-invariant. Fix $c \in V$. Let $\eps \in (0,r)$ be such that $[x,c] \subset V$ for all $x \in B(c,\eps)$. Fix $x,y \in B(c,\eps)$. Let $\lambda_0 \in (0,1]$ be such that, for any $\lambda \in [\lambda_0,1]$ and any $t \in [0,1]$, there exist two maximal cells $A,B$ of $X$ such that $A \cup B$ contains a neighbourhood of $\sigma_{c+\lambda x,c+\lambda y}(t)$ and of $c+\lambda \sigma_{x,y}(t)$. This implies that the path $c+\lambda \sigma_{x,y}$ locally coincides with $\sigma$. In particular, we deduce that the function $t \mapsto d(\sigma_{c+\lambda x,c+\lambda y}(t),c+\lambda \sigma_{x,y}(t))$ is locally convex, hence it is identically $0$. This implies the desired local dilation invariance property.

For any Busemann-convex finite-dimensional normed space, the geodesic bicombing is given by constant speed affine segments. This implies that for any face $C$ of $X$,  the restriction of $\sigma$ to $(C \cap V) \times (C \cap V)$ is the constant speed affine bicombing. 

The last condition on the union of two maximal cells is simpler to state, but implies the similar condition of Theorem~\ref{thm:gluing_busemann_convex}. We can therefore apply Theorem~\ref{thm:gluing_busemann_convex} (as stated in Remark~\ref{rk:gluing_busemann_convex}), and deduce that $\sigma$ is convex.

Proposition~\ref{pro:locally_convex_busemann} now tells us that $X$ is locally Busemann-convex. Because it is simply connected, Theorem~\ref{thm:cartan_hadamard} shows that $X$ is globally Busemann-convex. 
\ep

\brk
Note that the existence of a local consistent geodesic bicombing follows, for instance, from the assumption that the complex $X$ is locally uniquely geodesic. However, proving local unique geodesicity is usually hard. In fact, since Busemann-convexity implies unique geodesicity, Theorem~\ref{thm:criterion_busemann_convex} can also be seen as a way to prove unique geodesicity starting from a local consistent geodesic bicombing.
\erk

\section{Examples and counterexamples} \label{sec:examples}

The motivating idea for considering $\ell^p$ metrics is that, although it can happen that a cell complex fails to be uniquely geodesic (and hence cannot be CAT(0)) when the cells are given the $\ell^2$ metric, in some cases increasing the value $p$ can alleviates the problem. In this section we describe some simple examples where this happens, and some counterexamples.

\mk

One of the main sources of examples of piecewise $\ell^\infty$ complexes with convex bicombings come from orthoscheme complexes of lattices, see~\cite{haettel:lattices} and \cite{haettel:link}. Let us recall a few basic definitions.

\mk

A poset $L$ is \emph{graded} if, for each $x<y$ in $L$, all maximal chains from $x$ to $y$ have the same length. A graded poset has rank at most $r \in \N$ if there exists a function $\rk:L \ra \{0,1,\dots,r\}$ such that, for each $x<y$, one has $\rk(x)<\rk(y)$.

A poset $L$ is \emph{bounded} if it has both a minimum and a maximum.

A poset $L$ is \emph{flag} if any three $a,b,c \in L$ which are pairwise upper bounded have an upper bound.

A poset $L$ is a \emph{lattice} if any two $x,y \in L$ have a minimal upper bound $x\vee y$, called their \emph{join}, and a maximal lower bound $x\wedge y$, called their \emph{meet}.

A poset $L$ is a (meet)-\emph{semilattice} if one only requires the existence of meets.

\mk

The geometric realisation of a poset $L$ is the simplicial complex $|L|$ with vertex set $L$, and whose simplices are chains of $L$. Such a geometric realisation can be endowed with several natural metrics.

\mk

Consider an $n$-dimensional simplex $\sigma$ of $|L|$, whose vertices correspond to a chain $v_0<v_1< \dots <v_n$ in $L$, which is a maximal chain between $v_0$ and $v_n$. Then one can naturally identify $\sigma$ with the standard $n$-simplex of type $\tilde{C}_n$, which is also called the standard orthosimplex. It may be defined as the convex hull in $\R^n$ of the set of points $v_0=(0,0,\dots,0),v_1=(1,0,\dots,0),\dots,v_n=(1,1,\dots,1)$, see Figure~\ref{fig:3_orthoscheme}. It also coincides with a simplex of the barycentric subdivision of the $n$-cube $[0,2]^n$. For any $p \in [1,\infty]$, one may endow this orthosimplex with the standard $\ell^p$ metric of $\R^n$. The resulting length metric on the geometric realisation $|L|$ is called the \emph{$\ell^p$ orthoscheme metric}.

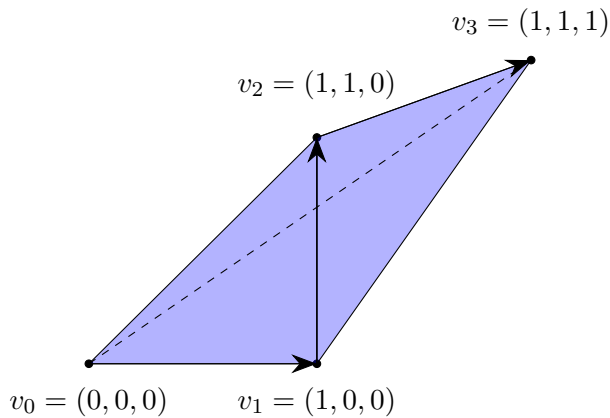
\begin{figure}[ht]
\centering
\begin{tikzpicture}
\def \p {0.05}
\def \op {0.3}
\def \gris {blue}
\draw[fill] (0,0) circle (\p) node(0) {};
\draw[fill] (3,0) circle (\p) node(1) {};
\draw[fill] (3,3) circle (\p) node(2) {};
\draw[fill] (3,3) + (20:3) circle (\p) node(3) {};

\draw[black,fill opacity=\op,fill=\gris] (0.center) -- (1.center) -- (2.center) -- (0.center);
\draw[black,fill opacity=\op,fill=\gris] (1.center) -- (2.center) -- (3.center) -- (1.center);
\draw[dashed] (0.center) -- (3.center);
\draw[-{Stealth[scale=2]}] (0.center) -- (1.center);
\draw[-{Stealth[scale=2]}] (1.center) -- (2.center);
\draw[-{Stealth[scale=2]}] (2.center) -- (3.center);

\node at ([yshift=-0.5cm]0) {\bfseries $v_0=(0,0,0)$};
\node at ([yshift=-0.5cm]1) {\bfseries $v_1=(1,0,0)$};
\node at ([yshift=0.7cm]2) {\bfseries $v_2=(1,1,0)$};
\node at ([yshift=0.5cm]3) {\bfseries $v_3=(1,1,1)$};

\end{tikzpicture}
\caption{The standard $3$-orthosimplex.}
\label{fig:3_orthoscheme}
\end{figure}

\bthm[{\cite[Thm~6.1]{haettel:lattices}}] \label{thm:orthoscheme_linfini}
Let $L$ denote a graded flag semilattice with minimum. The orthoscheme complex of $L$, endowed with the $\ell^\infty$ orthoscheme metric, is injective and admits a unique convex geodesic bicombing. 
\ethm

Here is a natural question.

\begin{question}
Consider a graded poset $L$, with orthoscheme complex $X$ endowed with the $\ell^p$ orthoscheme metric $d_p$. Under which conditions on $L$ and $p$ is the metric space $(X,d_p)$ uniquely geodesic? When is it Busemann-convex?
\end{question}

Note that, in order to also include the case $p=\infty$ (which is not uniquely geodesic), the correct formulation would be to determine when $(X,d_p)$ admits a unique convex bicombing, i.e. when it is a CUB space. For $p=\infty$, Theorem~\ref{thm:orthoscheme_linfini} provides a simple sufficient condition. For $p=2$, it is equivalent to asking for the CAT(0) property, and the only sufficient known condition is when $L$ is a modular semilattice (see~\cite{chalopinchepoihiraiosajda:weakly}, \cite{hirai:nonpositive} and \cite{haettelkielakschwer:strand}). However, some bounded graded lattices are not CAT(0), even though they are injective. The next result suggests that this phenomenon occurs more at higher ranks.

\bpro
Let $L$ denote a rank $2$ graded flag meet-semilattice with minimum, or a rank $3$ bounded graded lattice. Then for any $p \in [2,\infty]$, the orthoscheme complex $X$ of $L$ with the $\ell^p$ metric admits a unique convex geodesic bicombing.
\epro

\bp
Let us consider the case of a bounded graded lattice $L$ of rank $3$, and fix $p \in [2,\infty)$. Assume by contradiction that $x,y \in X$ are such that there exist two $\ell^p$ geodesics between $x$ and $y$. Then $x,y$ both lie in the subcomplex $Y$ of $X$ corresponding to a loop of $6$ elements $a,ab,b,bc,c,ca$ in $L$, with $a,b,c$ of rank $1$ and $ab,bc,ca$ of rank $2$, with the obvious order relations. This subcomplex $Y$ is actually locally isometric to $(\R^3,\ell^p)$, so this contradicts that there were two geodesics between $x$ and $y$. The case of a meet-semilattice is similar, with a loop of $8$ elements instead.
\ep

We are mostly interested in examples where the $\ell^2$ metric is not CAT(0), but the $\ell^p$ metric is uniquely geodesic for $p$ sufficiently large. We may conjecture that, if $2 \leq p < p' \leq \infty$, then $(X,d_p)$ uniquely geodesic implies $(X,d_{p'})$ uniquely geodesic (or has a unique convex bicombing in the case $p'=\infty$). We start with a simple, two-dimensional example which can be realised inside a semilattice.

\bexe
Consider the affine triangle (orthosimplex) $T$ in $\R^{2n+1}$ with vertices $o=(0,\dots,0)$, $(\overbrace{1,\dots,1}^{n},\overbrace{0,0,\dots,0}^{n+1})$, and $(\overbrace{1,\dots,1}^{n+1},\overbrace{0,0,\dots,0}^{n})$ endowed with the $\ell^p$ metric. Let $X$ denote the union of ten alternating copies of $T$ cyclically arranged around $o$, and let $d_p$ denote the length metric associated to the $\ell^p$ metric on $\R^{2n+1}$.

When $p=2$, the angle at $o$ of the triangle $T$ equals $\cos^{-1} \f{n}{\sqrt{n(n+1)}}$. For $n \geq 2$, this angle is smaller than $\f{2\pi}{10}$, and so $(X,d_2)$ is not CAT(0).

For a fixed value of $n$, as $p$ goes to $\infty$ the metric space $(X,d_p)$ converges to $(X,d_\infty)$ (in the Gromov-Hausdorff topology), which is injective. Indeed, when $p=\infty$ the triangle $T$ is isometric to the triangle in $(\R^2,\ell^\infty)$ with vertices $(0,0),(1,0),(1,1)$, independently of $n$. We see that, when $p=\infty$, every geodesic is contained in a union of at most four triangles. As a consequence, if $x,y \in X$ are such that there is some $p_0$ with the property that for all $p\ge p_0$ there exists an $\ell^p$ geodesic $\gamma_p$ from $x$ to $y$ that does not pass through $o$, then it must be the case that $\gamma_p$ is contained in a union of at most four triangles. Thus, if $p$ is large enough, the metric space $(X,d_p)$ is uniquely geodesic.
\eexe

A similar $3$--dimensional example can be realised inside a lattice.

\bexe
Consider the $3$--dimensional orthosimplex $T$ in $\R^7$ with vertices $o=(0,\dots,0)$, $(1,1,1,0,0,0,0)$, $(1,1,1,1,0,0,0)$, and $e=(1,\dots,1)$. Let $X$ denote the union of $8$ alternating copies of $T$ cyclically arranged around the diagonal edge $[o,e]$, and let $d_p$ denote the length metric associated to the $\ell^p$ metric. For $p=2$, the metric space $(X,d_2)$ is not CAT(0). But for $p$ large enough, $(X,d_p)$ is uniquely geodesic.
\eexe

However, it is not true that for every bounded graded lattice $L$, its orthoscheme complex with the $\ell^p$ metric becomes uniquely geodesic for $p<\infty$ large enough.

\bexe
Let us consider the orthoscheme complex $X$ of the following bounded lattice of rank $4$: it has three rank $1$ elements denoted $a,b,c$, three rank $2$ vertices denoted $ab,ac,bc$, and three "artifical" rank $3$ vertices denoted $\tilde{ab},\tilde{ac},\tilde{bc}$. It also has a minimum element $0$, and a maximum element $1$ of rank $4$. The partial order is the natural one, with covering relations between rank $2$ and rank $3$ elements being $ab<\tilde{ab}$, $ac<\tilde{ac}$ and $bc<\tilde{bc}$. We shall see that the $\ell^p$ metric on $X$ is not uniquely geodesic for any $p<\infty$. In order to prove this, it is sufficient to show that in the full subcomplex $Y$ of $X$ with vertices $0,1,a,ab,b,bc$, the $\ell^p$ geodesic from $a$ to $bc$ does not intersect the diagonal $[0,1]$. With coordinates, $Y$ can be represented as a subcomplex of $\R^4$, with vertices $0=(0,0,0,0)$, $1=(1,1,1,1)$, $a=(1,0,0,0)$, $ab=(1,1,0,0)$, $b=(0,1,0,0)$ and $bc=(0,1,1,0)$.

A point $z$ in the triangle spanned by $0,1,b$ has coordinates $z=(x,y,x,x)$, with $0 \leq x \leq y \leq 1$. The $\ell^p$ distance from $a$ to $bc$ via $z$ is
$$\left((1-x)^p+y^p+2x^p\right)^\f{1}{p}+\left((1-x)^p+(1-y)^p+2x^p\right)^\f{1}{p}.$$
If we minimise this quantity over all $(x,y) \in [0,1]^2$, we get $y=\f{1}{2}$ by symmetry, and then $x=\left(1+2^\f{1}{p-1}\right)^{-1} < \f{1}{2}=y$. So the geodesic from $a$ to $bc$ in $Y$ does not intersect the diagonal $[0,1]$. By symmetry, there is another geodesic from $a$ to $bc$ through the other triangle $0,1,c$. Hence, for any $p<\infty$, the metric space $(X,d_p)$ is not uniquely geodesic. 
\eexe

Nonetheless, it is true in this example that \emph{the union of all} $\ell^p$ geodesics converges to the $\ell^\infty$ convex bicombing as we increase $p$.

\section{Smoothness, convexity, and bolicity} \label{sec:bolicity}

Throughout this section, let $X$ be a metric space with a convex, consistent, geodesic bicombing $\sigma$. Our goal in this section is to use $\sigma$ to prove a local-to-global criterion for strong bolicity of $X$. To facilitate this, we shall consider two fine metric properties. As discussed in the introduction, these are \emph{uniform convexity} and \emph{uniform smoothness}, and our definitions are inspired by work of Ohta (see~\cite[\S~8.3]{ohta:comparison} and \cite[Prop.~5.4]{ohta:uniform}). 

In the case of a finite-dimensional normed vector space, uniform smoothness of the metric essentially amounts to asking that the unit ball is uniformly smooth, and uniform convexity essentially amounts to asking that the unit ball is uniformly convex. As noted in \cite[p.275]{bucherkarlsson:ondefinition}, for Banach spaces being strongly bolic is equivalent to being uniformly convex and uniformly smooth.

\subsection{Strong bolicity}

The notion of bolicity was introduced by Kasparov--Skandalis in relation with their work on the Novikov conjecture~\cite{kasparovskandalis:groupes,kasparovskandalis:groups}. Lafforgue then defined a strengthening of this notion, called strong bolicity, in his work on the Baum--Connes conjecture \cite{lafforgue:theorie}. The main result motivating the study of strongly bolic metric spaces is the following result of Lafforgue.

\bthm[{\cite{lafforgue:theorie}}] \label{thm:baum_connes}
If $G$ is a finitely generated group that:
\bit
\item   has the rapid decay property (see~\cite{chatterji:introduction}).
\item   acts properly by isometries on a uniformly locally finite strongly bolic metric space,
\eit
then $G$ satisfies the Baum--Connes conjecture without coefficients.
\ethm

Let us recall the definition of a strongly bolic metric space. For an interesting discussion of this and related notions, see~\cite{bucherkarlsson:ondefinition}.

\begin{defi}[Strong bolicity] \label{def:strong_bolicity}
A metric space $X$ with a geodesic bicombing $\sigma$ is
\bit
\item \emph{strongly B1} if for all $\delta,r>0$ there exists $R=R(\delta,r) \geq 0$ such that, for all $a,a',b,b' \in X$ with $d(a,b),d(a',b'),d(a,b'),d(a',b) \geq R$ and $d(a,a'),d(b,b') \leq r$, we have
\[
d(a,b) + d(a',b') - d(a,b') - d(a',b) \leq \delta;
\]
\item \emph{weakly B2} if for all $C>0$ there exists $N=N(C)>0$ such that, for all $x,y,z \in X$ with $d(x,y),d(x,z) \leq N$ and $d(y,z)>N$, we have 
\[
d(x,\sigma_{yz}(\f12)) < N-C.
\] 
\eit
If $X$ is both strongly B1 and weakly B2, then we say that $X$ is \emph{strongly bolic}. See Figures~\ref{fig:bolicity_B1} and \ref{fig:bolicity_B2}.
\end{defi}

\begin{figure}[ht]
\centering
\begin{tikzpicture}
\def \p {0.05}
\def \op {0.3}
\def \gris {blue}
\draw[fill] (-3,0) circle (\p) node(a) {};
\draw[fill] (-3.2,0.5) circle (\p) node(a') {};
\draw[fill] (4,-0.3) circle (\p) node(b) {};
\draw[fill] (4.5,1) circle (\p) node(b') {};

\draw[blue] (a.center) -- (b.center);
\draw[blue] (a'.center) -- (b'.center);
\draw[green] (a'.center) -- (b.center);
\draw[green] (a.center) -- (b'.center);

\draw[<->] (-2,-0.5) -- (3,-0.5);
\node at (0.5,-1) {$\geq R$};
\node at ([yshift=-0.5cm]a) {$a$};
\node at ([yshift=0.5cm]a') {$a'$};
\node at ([yshift=-0.5cm]b) {$b$};
\node at ([yshift=0.5cm]b') {$b'$};

\draw[<->] (-3.5,0) -- (-3.5,0.5);
\node at (-4,0.25) {$r \geq$};

\draw[<->] (5,-0.3) -- (5,1);
\node at (5.5,0.4) {$\leq r$};

\end{tikzpicture}
\caption{The B1 bolicity condition: smoothness.}
\label{fig:bolicity_B1}
\end{figure}
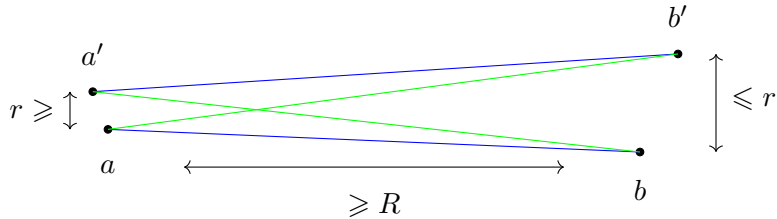

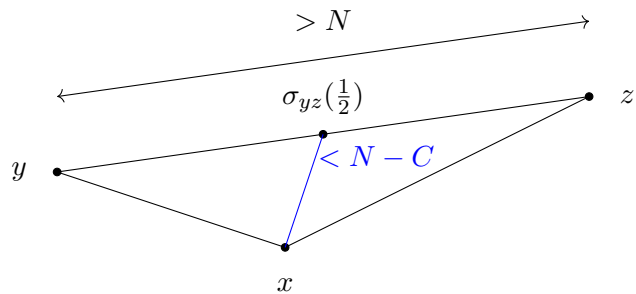
\begin{figure}[ht]
\centering
\begin{tikzpicture}
\def \p {0.05}
\def \op {0.3}
\def \gris {blue}
\draw[fill] (0,0) circle (\p) node(x) {};
\draw[fill] (-3,1) circle (\p) node(y) {};
\draw[fill] (4,2) circle (\p) node(z) {};
\draw[fill] (0.5,1.5) circle (\p) node(m) {};

\draw (y.center) -- (x.center) -- (z.center) -- (y.center);
\draw[blue] (x.center) -- (m.center);

\node at ([yshift=-0.5cm]x) {$x$};
\node at ([xshift=-0.5cm]y) {$y$};
\node at ([xshift=0.5cm]z) {$z$};
\node at ([yshift=0.5cm]m) {$\sigma_{yz}(\f12)$};

\draw[<->] (-3,2) -- (4,3);
\node at (0.5,3) {$> N$};

\node at (1.2,1.2) {$\color{blue} < N-C$};

\end{tikzpicture}
\caption{The B2 bolicity condition: convexity.}
\label{fig:bolicity_B2}
\end{figure}

Note that every CAT(0) metric space is strongly bolic, and it turns out to be the most common source of examples. One very interesting example is the following: consider a Gromov-hyperbolic group $G$, with Cayley graph $X$. Then, according to Mineyev--Yu \cite{mineyevyu:baumconnes}, there exists a strongly bolic $G$--equivariant metric on $X$. Another much simpler construction is given by the Green distance associated to a symmetric finite support random walk on $G$, see~\cite{haissinskymathieu:conjecture}.

Busemann-convexity and strong bolicity are independent properties. In one direction, strong bolicity passes to arbitrary subsets, some of which are non-geodesic. More explicitly, one may consider one of the two strongly bolic metrics discussed above on a non-free hyperbolic group, which does not even have unique combinatorial geodesics.

In the other, consider the infinite wedge sum $\bigvee_{p\in(1,\infty)}(\R^2,\ell^p)$: it is Busemann-convex, but one cannot find uniform strong bolicity constants. As another example, consider a norm $N$ on $\R^2$ whose unit ball is strictly convex but not smooth (e.g. $N=\ell^1+\ell^2$). Then $(\R^2,N)$ is Busemann-convex, but is not strongly bolic.

\brk \label{rem:bolicity}
Originally, Kasparov--Skandalis defined a metric space to be \emph{bolic} if it satisfies the \emph{B1} condition (the strong B1 condition holds for some $\delta$) and what we shall call the \emph{$\ell^2$--B2} condition, which is different to the weak B2 condition above. As they note in \cite[Rem.~2.8]{kasparovskandalis:groups}, this is ``very much a Euclidean condition''. The results of their paper are all proved for spaces satisfying B1 and weak B2.

Definition~\ref{def:strong_bolicity} is never explicitly referred to by Lafforgue as strong bolicity \cite{lafforgue:theorie}. The term ``strongly bolic'' appears first in Mineyev--Yu \cite{mineyevyu:baumconnes}. They state that Lafforgue proves the Baum--Connes conjecture for strongly bolic groups, but later define strong bolicity as being strongly B1 and satisfying $\ell^2$--B2. In their setting of hyperbolic groups, these are equivalent. In the general metric setting, though, the B2 condition is more appropriate than the $\ell^2$--B2 condition, which is why we have reverted to Lafforgue's setting.

The only case we are aware of where one really needs the $\ell^2$--B2 condition instead of the B2 condition is for Kar's result that bolic spaces (in the $\ell^2$ sense) are \emph{asymptotically CAT(0)} \cite{kar:discrete}. This is not true for the plane with an $\ell^p$ metric for $p\ne2$.
\erk

\subsection{Uniform convexity} \label{subsec:uniform_convexity}

The following is an $\ell^p$ modification of \cite[Def.~8.9]{ohta:comparison}, and the corresponding notion of smoothness is considered in Section~\ref{subsec:smoothness}.

\begin{defi}[$(p,k)$--uniform convexity] \label{def:uniform_convex}
Let $p\in(1,\infty)$ and let $k > 0$. We say that $(X,\sigma)$ is \emph{$(p,k)$--uniformly convex} if
\[
d(x,\sigma_{yz}(\f12))^p \,\leq\, \f{1}{2}d(x,y)^p + \f{1}{2}d(x,z)^p - kd(y,z)^p
\]
holds for all $x,y,z\in X$.
\end{defi}

Observe that if $(X,\sigma)$ is $(p,k)$--uniformly convex, then $X$ is $(p,k')$--uniformly convex for all $k'\le k$, so we can always assume that $k\in(0,1)$. The following proposition connects uniform convexity with strong bolicity.

\bpro \label{prop:convex_B2}
If $(X,\sigma)$ is $(p,k)$-uniformly convex, then it is weakly~$B2$.
\epro

\bp
As noted above, we may assume that $k<1$. Given $C>0$, let $N>0$ be sufficiently large that $(1-k)^{\f{1}{p}} < 1-\f{C}{N}$. By definition of $(p,k)$--convexity, if $x,y,z \in X$ are such that $d(x,y),d(x,z) \leq N$ and $d(y,z)>N$, then we have 
\[
d(x,\sigma_{yz}(\f12))^p \,\leq\, \f12N^p+\f12N^p - kN^p \,=\, (1-k)N^p.
\]
We thus have $d(x,\sigma_{yz}(\f12)) \leq (1-k)^{\f1p}N < N-C$.
\ep

The following is essentially an adaptation of~\cite[Prop.~5.4]{ohta:uniform} to our precise definitions.

\begin{pro} \label{prop:convexity_local_to_global}
Let $X$ be a metric space with a convex, consistent bicombing $\sigma$. If $(X,\sigma)$ is locally $(p,k)$--uniformly convex, then it is globally $(p,k)$--uniformly convex.
\end{pro}

\begin{proof}
Given $x,y,z \in X$, let us write $m=\sigma_{yz}(\f12)$. Since $X$ is locally $(p,k)$--uniformly convex, there exists $\eps>0$ such that the ball $B(m,\eps)$ is $(p,k)$--uniformly convex.

Fix $\alpha \in (0,1]$ small enough such that $\alpha d(m,x)$, $\alpha d(m,y)$, and $\alpha d(m,z)$ are all less than $\eps$. Then let us denote $x'=\sigma_{mx}(\alpha)$, $y'=\sigma_{my}(\alpha)$, and $z'=\sigma_{mz}(\alpha)$, all of which lie in $B(m,\eps)$. Note that since $\sigma$ is consistent, we also have $m=\sigma_{y'z'}(\f12)$.

By $(p,k)$--uniform convexity of $B(m,\eps)$, we know that
\[
d(x',m)^p \,\leq\, \f{1}{2}d(x',y')^p + \f{1}{2}d(x',z')^p - kd(y',z')^p.
\]
By convexity of $\sigma$, and since $d(y',z')=\alpha d(y,z)$, we deduce that
\begin{align*}
d(x,m)^p \,&=\, \alpha^{-p} d(x',m)^p \\
	&\leq\, \alpha^{-p}\left(\f{1}{2}d(x',y')^p + \f{1}{2}d(x',z')^p - kd(y',z')^p\right) \\
	&\leq\, \alpha^{-p}\left(\f{1}{2}\alpha^p d(x,y)^p + \f{1}{2} \alpha^p d(x',z')^p - k\alpha^p d(y,z)^p\right) \\
	&\leq\, \f{1}{2}d(x,y)^p + \f{1}{2}d(x,z)^p - kd(y,z)^p,
\end{align*}
which shows that $X$ is globally $(p,k)$-uniformly convex.
\end{proof}

In the case of piecewise $\ell^p$ complexes, we obtain $(p,k)$--convexity of the cells from the following classical results of Clarkson \cite{clarkson:uniformly} and Hanner \cite{hanner:onuniform}, see also~\cite{ballcarlenlieb:sharp}.

\begin{lem} \label{lem:lp_convex}
For every $p \in [2,\infty)$, there exists a constant $k_p=\f1{2^p}$ such that for every $n \geq 0$, the normed vector space $(\R^n,\ell^p)$ is $(p,k_p)$--uniformly convex.

For every $p \in (1,2]$, there exists a constant $k_p>0$ such that for every $n \geq 0$, the normed vector space $(\R^n,\ell^p)$ is $(2,k_p)$--uniformly convex.
\end{lem}

\bpro \label{prop:cell_convex}
Assume that $X$ is a locally finite-dimensional cell complex such that each cell is endowed with a $(p,k)$--uniformly convex norm. Assume that $X$, with the induced length metric, has a consistent, convex bicombing $\sigma$. Then $(X,\sigma)$ is $(p,k)$--uniformly convex.
\epro

\bp
According to Proposition~\ref{prop:convexity_local_to_global}, it suffices to prove that $X$ is locally $(p,k)$--uniformly convex. Fix a point $u \in X$, and let $F$ denote the minimal closed cell containing $u$. We shall prove that $X$ is locally $(p,k)$--uniformly convex at $u$ by induction on the codimension of $F$ in $X$. If $F$ is a maximal cell, then by assumption the interior of $F$ is locally $(p,k)$-uniformly convex.

Assume that the result is proved for cells with codimension at most $q$, and suppose that $F$ has codimension $q+1$. Given $x,y,z \in X$ lying in a neighbourhood of $F$, there exists sequences $(y_n)_{n \geq 0}$ and $(z_n)_{n \geq 0}$ converging to $y$ and $z$ respectively, such that $m_n=\sigma_{y_n,z_n}(\f12)$ lies in a cell $F_n$ of codimension $\leq q$ for all $n \in \N$. For each $n \in \N$, there exists $\alpha_n>0$ small enough such that $\sigma_{m_n,x}(\alpha_n)$, $\sigma_{m_n,y_n}(\alpha_n)$ and $\sigma_{m_n,z_n}(\alpha_n)$ all lie in a $(p,k)$-uniformly convex neighbourhood of $m_n$, by induction. Hence, as in the proof of Proposition~\ref{prop:convexity_local_to_global}, we deduce that
\[
d(x,m_n)^p \,\leq\, \f{1}{2}d(x,y_n)^p + \f12d(x,z_n)^p - kd(y_n,z_n)^p.
\]
Since this holds for all $n \in \N$, we deduce by considering the limit as $n \ra \infty$ that
\[
d(x,m)^p \,\leq\, \f{1}{2}d(x,y)^p + \f{1}{2}d(x,z)^p - kd(y,z)^p,
\]
which shows that $X$ is $(p,k)$-uniformly convex.
\ep

\subsection{Uniform smoothness} \label{subsec:smoothness}

We now turn to uniform smoothness of metric spaces. The proofs are similar to those of Section~\ref{subsec:uniform_convexity}, as is to be expected by duality in the linear setting. Let us begin by recalling the notion of uniform smoothness for Banach spaces (see \cite[Def.~1.e.1]{lindenstrausstzafriri:classical:2}, for instance).

\begin{defi}[Uniform smoothness] \label{def:uniform_smooth}
Let $Y$ be a Banach space. The \emph{modulus of smoothness} of $Y$ is the function 
\[
\rho_Y(\tau) \,=\, \sup\left\{\f{\|z+y\|+\|z-y\|}{2\|z\|}-1 \,:\, y,z \in Y \bs \{0\}, \|y\|=\tau\|z\|\right\}.
\]
We say $Y$ is \emph{uniformly smooth} if $\lim_{\tau\to0}\f{\rho_X(\tau)}\tau=0$. 
\end{defi}

One can consider spaces where the modulus of smoothness decays more quickly than linearly---we shall be interested in quadratic decay. This idea first appears in \cite{ballcarlenlieb:sharp} as the following.

\begin{defi}[2--uniform smoothness, {\cite[p.~468]{ballcarlenlieb:sharp}}]
A Banach space $Y$ is said to be \emph{2--uniformly smooth} if there is some $K$ such that
\[
\f{\|z+y\|^2+\|z-y\|^2}2\,\le\,\|z\|^2+\|Ky\|^2
\]
holds for all $y,z\in Y$.
\end{defi}

One can straightforwardly compute
\begin{align*}
\f{\|z+y\|+\|z-y\|}{2\|z\|} \,
    &\le\, \f1{\|z\|}\left(\f{\|z+y\|^2+\|z-y\|^2}2\right)^{\f12} \\
    &\le\, \left(1+\f{\|Ky\|^2}{\|z\|^2}\right)^{\f12} \\
    &\le\, 1+\f{K^2}2\f{\|y\|^2}{\|z\|^2},
\end{align*}
which shows that $\rho_Y(\tau)\le\f{K^2}2\tau^2$. In fact, according to \cite[Prop.~7]{ballcarlenlieb:sharp}, 2--uniform smoothness is equivalent to quadratic decay of the modulus of smoothness.

Returning to the setting of a (not necessarily linear) metric space $X$ with a convex, consistent bicombing $\sigma$, we consider the following adaptation of 2--uniform smoothness, for some constant $C>0$.

\begin{defi}[$(2,C)$--uniform smoothness]
We say that $(X,\sigma)$ is \emph{$(2,C)$--uniformly smooth} if, for every $r,R>0$ with $R \geq 2r$, the inequality
\[
d(x,\sigma_{yz}(\f12)) \leq d(x,z) - \f12d(y,z) + \f{Cr^2}{R}
\]
is satisfied by every triple $x,y,z \in X$ with $d(x,y) \leq r$ and $d(y,z) \geq R$.
\end{defi}

\begin{lem} \label{lem:2smooth_to_2Csmooth}
Let $(X,\sigma)$ be a Banach space, with $\sigma$ being the natural affine bicombing. If $X$ is 2--uniformly smooth with constant $K$, then it is $(2,\f{K^2}4)$--uniformly smooth.
\end{lem}

\begin{proof}
Suppose that $x,y,z\in X$ have $d(x,y)\le r$ and $d(y,z)\ge R$ for some $r,R>0$ with $R\ge2r$. Using the linear structure of $X$, let $y'=y-x$, $z'=z-x$. Let $\tau=\f{\|y'\|}{\|z'\|}$. The various definitions now give 
\[
\dist(x,\sigma_{yz}(\f12)) \,=\, \f{\|z'+y'\|}2 \,\le\, \|z'\|\left(\rho_X(\tau) - \f{\|z'-y'\|}{2\|z'\|}+1\right).
\]
By the bound on the modulus of smoothness following from 2--smoothness and the fact that $R\ge2r$, we see that 
\begin{align*}
d(x,\sigma_{yz}(\f12)) - d(x,z) + \f12d(y,z) \,
    &=\,   \f{\|z'+y'\|}2 - \|z'\|\left(1 - \f{\|z'-y'\|}{2\|z'\|}\right) \\
    &\le\, \|z'\|\rho_X(\tau) \\    
    &\le\, \f{K^2}2\f{d(x,y)^2}{d(x,z)} \\
    &\le\, \f{K^2}2\f{r^2}{R-r} \,\le\, \f{K^2}4\f{r^2}R. \qedhere
\end{align*}
\end{proof}

According to Lemma~\ref{lem:lp_convex}, for every $q \in (1,2]$, there exists $k>0$ such that $(\R^n,\ell^q)$ is $(2,k)$--uniformly convex. By duality, this implies that $(\R^n,\ell^p)$ is $2$--uniformly smooth (with associated constant depending on $k$, where $\f1p+\f1q=1$), hence $(2,C)$--uniformly smooth by Lemma~\ref{lem:2smooth_to_2Csmooth}. In fact one can obtain a more precise constant.

\blem \label{lem:lp_smooth}
For every $p \in [2,\infty)$ and every $n \geq 0$, the normed vector space $(\R^n,\ell^p)$ is $(2,\f{(p-1)^2}4)$--uniformly smooth.
\elem

\begin{proof}
According to \cite[Prop.~3]{ballcarlenlieb:sharp}, for any $p\ge2$, the space $(\R^n,\ell^p)$ is 2--uniformly smooth, with constant $K=p-1$. The result follows from Lemma~\ref{lem:2smooth_to_2Csmooth}.
\end{proof}

Note that, when $p<2$, the space $(\R^2,\ell^p)$ is not $2$--uniformly smooth (see~\cite[Prop.~3]{ballcarlenlieb:sharp}): indeed the modulus of smoothness decays like $\tau^p$. This behaviour is incompatible with the scaling argument used in Proposition~\ref{prop:smooth_local_to_global}.


The following uses essentially the same idea as in~\cite[Prop.~5.4]{ohta:uniform}, similarly to the case of uniform convexity in Proposition~\ref{prop:convexity_local_to_global}.

\bpro \label{prop:smooth_local_to_global}
Let $X$ be a metric space with a convex, consistent, bicombing $\sigma$. If $(X,\sigma)$ is locally $(2,C)$--uniformly smooth, then it is globally $(2,C)$--uniformly smooth.
\epro
 
\bp
Given $x,y,z \in X$ satisfying $d(x,y)\le r$ and $d(y,z)\ge R$, where $R\ge2r$, let us denote $m=\sigma_{yz}(\f12)$. Since $X$ is locally $(2,C)$--uniformly smooth, there exists $\eps>0$ such that $B(m,\eps)$ is $(2,C)$--uniformly smooth. 

Fix $\alpha \in (0,1]$ small enough such that $\alpha d(m,x)$, $\alpha d(m,y)$, and $\alpha d(m,z)$ are all less than $\eps$. Then $x'=\sigma_{m,x}(\alpha)$, $y'=\sigma_{m,y}(\alpha)$, and $z'=\sigma_{m,z}(\alpha)$ all lie in $B(m,\eps)$. Note that since $\sigma$ is a consistent bicombing, we have $m=\sigma_{y'z'}(\f12)$, and $d(y',z')=\alpha d(y,z)\ge\alpha R$. By convexity of $\sigma$, we also know that $d(x',y') \leq \alpha d(x,y) \leq \alpha r$. The $(2,C)$--uniform smoothness of $B(m,\eps)$ now tells us that
\[
d(x',m) \leq d(x',z') - \f12d(y',z')+\f{C\alpha^2 r^2}{\alpha R}.
\]
Convexity of $\sigma$ also gives $d(x',z')\le\alpha d(x,z)$. Again using the fact that $\sigma$ is a consistent bicombing, we conclude that
\begin{align*} 
d(x,m) \,&=\, \alpha^{-1} d(x',m) \\
    &\leq\, \alpha^{-1}(d(x',z') - \f12d(y',z')+\f{C\alpha r^2}R) \\
    &\leq\, \alpha^{-1}(\alpha d(x,z) - \f12\alpha d(y,z)+\f{C\alpha r^2}R) \\
    &=\, d(x,z) - \f{1}{2}d(y,z) + \f{Cr^2}{R},
\end{align*}
so $X$ is globally $(2,C)$--uniformly smooth.
\ep

\bpro \label{prop:cell_smooth}
Let $X$ be a locally finite-dimensional cell complex such that each cell is endowed with a $(2,C)$--uniformly smooth norm. If $X$, with the induced length metric, has a consistent, convex bicombing $\sigma$, then $(X,\sigma)$ is $(2,C)$--uniformly smooth.
\epro

\bp
According to Proposition~\ref{prop:smooth_local_to_global}, it suffices to prove that $X$ is locally $(2,C)$--uniformly smooth. Fix a point $u \in X$, and let $F$ denote the minimal closed cell containing $u$. We shall prove that $X$ is locally $(2,C)$--uniformly smooth at $u$ by induction on the codimension of $F$ in $X$. If $F$ is a maximal cell, then by assumption the interior of $F$ is locally $(2,C)$--uniformly smooth.

Assume that the result is proved for cells with codimension at most $q$, and suppose that $F$ has codimension $q+1$. Let $r>0$, let $R\ge2r$, and suppose that $x,y,z$ are points of $X$ lying in a neighbourhood of $F$ that satisfy $d(x,y)\le r$ and $d(y,z)\ge R$. There exists sequences $(y_n)_{n \geq 0}$ and $(z_n)_{n \geq 0}$ converging to $y$ and $z$ respectively, such that $m_n=\sigma_{y_n,z_n}(\f12)$ lies in a cell $F_n$ of codimension $\leq q$ for all $n \in \N$. Moreover, we can choose these sequences so that $R_n=d(y_n,z_n)\ge2d(x,y_n)=2r_n$. By induction, for each $n\in\N$ there exists $\alpha_n>0$ small enough that $x_n'=\sigma_{m_n,x}(\alpha_n)$, $y_n'=\sigma_{m_n,y_n}(\alpha_n)$, and $z_n'=\sigma_{m_n,z_n}(\alpha_n)$ all lie in a $(2,C)$--uniformly smooth neighbourhood of $m_n$. Note that by convexity of $\sigma$ we have 
\[
2d(x_n',y_n') \,\le\, 2\alpha_nd(x,y_n) \,=\, 2\alpha_nr_n 
    \,\le\, \alpha_nR_n \,\le\, \alpha_nd(y_n,z_n) \,=\, d(y_n',z_n').
\]
Hence, as in the proof of Proposition~\ref{prop:smooth_local_to_global}, we deduce that
\[
d(x,m_n) \leq d(x,z_n) - \f12d(y_n,z_n) + \f{Cr_n^2}{R_n}.
\]
Since this holds for all $n\in\N$, we deduce by considering the limit as $n\ra\infty$ that
\[
d(x,m(y,z)) \leq d(x,z) - \f12d(y,z) + \f{Cr^2}{R},
\]
which shows that $X$ is $(2,C)$--uniformly smooth.
\ep

Similarly to $(p,k)$--uniform convexity, $(2,C)$--uniform smoothness is related to bolicity, as we now make precise, strengthening \cite[Prop.~1]{bucherkarlsson:ondefinition}.

\bpro \label{prop:smooth_B1}
If $(X,\sigma)$ is $(2,C)$--uniformly smooth, then it is strongly B1.
\epro

\begin{proof}
For each $\delta>0$ and $r \geq 0$, let us define $R= \max\left\{\f{2Cr^2}{\delta},2r\right\}$. Suppose that $a,a',b,b' \in X$ satisfy $d(a,b),d(a,b'),d(a',b),d(a',b') \geq R$, $d(a,a'),d(b,b') \leq r$. Let $m=\sigma_{ab}(\f12)$. By $(2,C)$--uniform smoothness, we know that
\[
d(a',m) \,\leq\, d(a',b) - \f{1}{2}d(a,b)+\f{Cr^2}{R};
\]
\[
d(b',m) \,\leq\, d(b',a) - \f{1}{2}d(a,b)+\f{Cr^2}{R}.
\]
By summing these inequalities and using the fact that $d(a',b')\le d(a',m)+d(b',m)$, we conclude that
\[
d(a,b) + d(a',b') - d(a,b') - d(a',b) \,\le\, 2\f{Cr^2}R \,\le\, \delta. \qedhere
\]
\end{proof}

\subsection{Local-to-global for strong bolicity}

The following summarises the combination of Propositions~\ref{prop:cell_convex}, \ref{prop:convex_B2}, \ref{prop:cell_smooth}, and~\ref{prop:smooth_B1}.

\bthm \label{thm:convex_bicombing_implies_bolicity}
Let $X$ be a locally finite-dimensional cell complex where each cell is endowed with a $(p,k)$-uniformly convex, $(2,C)$-uniformly smooth norm. If, with the induced length metric, $X$ has a consistent, convex bicombing $\sigma$, then $X$ is $(p,k)$-uniformly convex, $(2,C)$-uniformly smooth, and strongly bolic.
\ethm

According to Lemmas~\ref{lem:lp_convex} and~\ref{lem:lp_smooth}, we can combine this theorem with the results from Section~\ref{sec:bicombing} in the case of an $\ell^p$-norm to obtain a simple local-to-global criterion for strong bolicity.

\bthm \label{thm:criterion_piecewise_lp_bolicity}
Let $X$ be a cell complex with finitely many shapes. Let $p\in[2,\infty)$ and give each cell an $\ell^p$ norm. Suppose that:
\bit
\item $X$ is simply connected;
\item $X$ locally admits a consistent geodesic bicombing;
\item for any two intersecting maximal cells $A,B$ of $X$, the union $A \cup B$, with the induced length metric, is Busemann-convex.
\eit
With the induced length metric, $X$ is Busemann-convex, $(p,k)$--uniformly convex, $(2,C)$--uniformly smooth, and strongly bolic.
\ethm

As previously discussed, finding a strongly bolic metric for a hyperbolic group is not straightforward. Nonetheless, Theorem~\ref{thm:criterion_piecewise_lp_bolicity} could potentially provide another strongly bolic model. Indeed, Lang proved that every hyperbolic group $G$ acts geometrically on an orthoscheme complex $X$, which is injective for the $\ell^\infty$ metric and admits a unique convex bicombing (see~\cite{lang:injective}). If the $\ell^p$ metric on $X$ is uniquely geodesic for some sufficiently large $p$ (depending on $G$), then this would be such a model. Note that if this complex $X$, endowed with $\ell^2$ metric, is locally uniquely geodesic, it would positively answer the famous open problem of whether hyperbolic groups are CAT(0).

\subsection{Splitting of centralisers} \label{subsec:splitting_centralisers}

One may wonder whether one could apply Lafforgue's Theorem~\ref{thm:baum_connes} to the mapping class groups of surfaces, which are known to have the rapid decay property \cite{behrstockminsky:centroids}. However, we shall show that any proper action of the mapping class group on a strongly bolic space would have to be highly distorted.

The main point is the simple observation that the splitting result for CAT(0) spaces, as stated in~\cite[Thm~6.12]{bridsonhaefliger:metric}, generalises to strongly B1 metric spaces. Note that the relationship between strong bolicity and horofunctions has already been observed by Haissinsky and Mathieu \cite{haissinskymathieu:conjecture}. The main rough idea is that, for a smooth metric space, the horofunction boundary coincides with the visual boundary.

\bthm[Generalisation of {\cite[Thm~6.12]{bridsonhaefliger:metric}}] \label{thm:centralisers_split}
Let $X$ be a strongly $B1$ metric space, and let $G$ be a finitely generated group acting by isometries on $X$. Assume that $G$ contains a central element $z$ such that $(z^n \cdot x_0)_{n \in \Z}$ is a quasi-geodesic in $X$, for some $x_0 \in X$. Then some finite-index subgroup of $G$ contains $\<z\>$ as a direct factor.
\ethm

\bp
Consider the horofunction bordification of $X$: it is the closure $\ov{X} = X \sqcup \partial X$ of the map
\beq X & \ra & {\mathcal F}(X,\R)/\R \\
x & \mapsto & [y \mapsto d(x,y)],\eeq
where $\R$ acts on the space ${\mathcal F}(X,\R)$ of functions by translation. We will denote by $[b]$ the translation class of a map $b:X \ra \R$.

Consider two sequences $(x_n)_{n \in \N}$ and $(y_n)_{n \in \N}$ in $X$ such that there exists $K \geq 0$ for which $d(x_n,y_n) \leq K$ for all $n\in\N$. According to the strong $B1$ property, if $(x_n)$ converges to $[b] \in \partial X$, then the sequence $(y_n)_{n \in \N}$ also converges to $[b]$. In particular, since $(z^n \cdot x_0)_{n \in \Z}$ is a quasi-geodesic in $X$, we deduce that the sequence $(z^n \cdot x_0)_{n \in \N}$ converges to some $[b] \in \partial X$ independently of $x_0$.

For any $g \in G$, since $z$ and $g$ commute, we deduce that the sequence $(gz^n \cdot x_0)_{n \in \N}$ converges to $[b]$, so $[g\cdot b]=[b]$. In particular, for each $g \in G$, there exists $r_g \in \R$ such that $g \cdot b=b + r_g$. The map $\phi: g \in G \mapsto r_g \in \R$ is a group homomorphism and $\phi(z) < 0$. Since $G$ is finitely generated, the image of $\phi$ is a finitely generated abelian group $\Z^d$, for some $d \geq 1$. Up to passing to a finite-index subgroup $G_0$ of $G$ containing $z$, and up to postcomposing $\phi$ with a homomorphism to $\Z$, we may assume that $\phi(G_0) = \Z$, and that $\phi(z)=-1$. We deduce that $G_0$ splits as $\Ker\phi \times \Z$.
\ep

As in the proof that mapping class groups do not act properly semisimply on CAT(0) spaces (\cite[Thm~7.26]{bridsonhaefliger:metric}, also see \cite[Thm~4.2]{kapovichleeb:actions}), we have the corresponding corollary mentioned above.

\bcor \label{cor:mcg_not_bolic}
For $g \geq 3$, the mapping class group $\Mod(S_g)$ has no proper action by isometries on a strongly bolic metric space inducing a quasi-isometric embedding.
\ecor

\bp
Following the proof of \cite[Thm~7.26]{bridsonhaefliger:metric}, consider the Dehn twist $z \in \Mod(S_g)$ along a separating simple closed curve in $S_g$ bounding a genus $1$ subsurface. Then the centraliser of $z$ in $\Mod(S_g)$ contains a subgroup $G$ isomorphic to a cocompact lattice in $\widetilde{\PSL(2,\R)}$. In particular, $G$ does not virtually split. 

Since $z$ has infinite order and is undistorted in $\Mod(S_g)$, the sequence $(z^n)_{n \in \Z}$ is a quasi-geodesic in $\Mod(S_g)$ with respect to some word metric. Assume that $\Mod(S_g)$ acts properly by isometries on a strongly bolic metric space inducing a quasi-isometric embedding. Then, for any $x_0 \in X$, the sequence $(z^n \cdot x_0)_{n \in \Z}$ is a quasi-geodesic in $X$. This contradicts Theorem~\ref{thm:centralisers_split}.
\ep

In particular, looking for proper semisimple actions of mapping class groups on piecewise $\ell^p$ complexes is hopeless.

\bcor \label{cor:mcg_not_lp}
For $g \geq 3$, the mapping class group $\Mod(S_g)$ does not act properly by semisimple isometries on a Busemann-convex, piecewise $\ell^p$ complex $X$ (where $2 \leq p < \infty$).
\ecor

\bp
According to Theorem~\ref{thm:criterion_piecewise_lp_bolicity}, $X$ satisfies the strong $B1$ property. Moreover, Busemann-convexity of $X$ means that the flat torus theorem holds \cite[Thm~1.2]{descombeslang:flats}, so the Dehn twist $z$ from the proof of Corollary~\ref{cor:mcg_not_bolic} is undistorted in any proper action of $\Mod(S_g)$ on $X$. We conclude in the same way as in Corollary~\ref{cor:mcg_not_bolic}.
\ep

It is interesting to contrast these results with \cite[Thm~1.2]{bestvinabrombergfujiwara:proper}, which states that mapping class groups admit proper actions on finite products of (locally infinite) quasi-trees such that orbits are quasi-isometric embeddings.

\section{\texorpdfstring{$\ell^p$}{lp}-metrics on CAT(0) cube complexes} \label{sec:lp_cube_complex}

The goal of this section is to show that every CAT(0) cube complex $X$ equipped with the piecewise $\ell^p$-metric is Busemann-convex. Unless explicitly stated otherwise, we will always consider the case $p \in (1,\infty)$. The main step is to show that $X$ is locally uniquely geodesic, and we then apply a local-to-global result. As part of our proof, we obtain a characterisation of local $\ell^p$-geodesics similar to \cite[Thm~5.8]{ardilaowensullivant:geodesics}, which in turn uses \cite{owenprovan:fast}. Also see \cite{hayashi:polynomial}.

Note that by Bridson's thesis, we already know that $X$ is (not necessarily uniquely) geodesic if it is finite-dimensional. Of course, we shall find \emph{a posteriori} from Busemann-convexity that $X$ is uniquely geodesic, even if it is not finite-dimensional or even locally finite-dimensional.

Also note that we will be considering CAT(0) cube complexes that need not be finite-dimensional, hence the results from Sections~\ref{sec:bicombing} and \ref{sec:bolicity} do not apply.

\subsection{Generalities on \texorpdfstring{$\ell^p$}{lp} metrics}

\begin{lem}[{\cite[Thm~1.1]{bridson:geodesics}}] \label{lem:geodesic}
Let $X$ be a finite-dimensional CAT(0) cube complex. If each cube is given the $\ell^p$-metric, then the induced length metric on $X$ is geodesic.
\end{lem}

We begin with a general lemma that will help simplify a number of arguments. It states that, on a product, the metric we are considering is just the product of the metrics we are considering on the factors.

\begin{lem} \label{lem:lp_product}
Let $p \in [1,\infty]$, and let $X_1$ and $X_2$ be connected cell complexes. Endow each cell of $X_i$ with the $\ell^p$ metric, and give $X_i$ the induced length metrics $d_i$. Consider the product $X=X_1 \times X_2$, where each cell of $X$ is given the associated $\ell^p$ metric from $X_1$ and $X_2$. The length metric $d$ on $X$ satisfies
\beq &d^p = d_1^p+d_2^p & \mbox{ if $p < \infty$} \\
&d = \max(d_1,d_2) & \mbox{ if $p=\infty$}.\eeq
\end{lem}

\begin{proof}
Let $x=(x_1,x_2),y=(y_1,y_2)\in X$. For $i=1,2$, let $\gamma_i:[0,1] \ra X_i$ be a constant-speed geodesic from $x_i$ to $y_i$ in $X_i$ with constant speed, so that $\|\gamma'_i(t)\| = d(x_i,y_i)$ for all $t$. If $p<\infty$, then the path $\gamma = (\gamma_1,\gamma_2) : [0,1] \ra X$ satisfies
\[
\|\gamma'(t)\| = (d(x_1,y_1)^p+(d(x_2,y_2)^p)^{\f{1}{p}} \text{ for all }t\in[0,1],
\]
so $d(x,y)^p \leq (d(x_1,y_1)^p+d(x_2,y_2)^p)^{\f{1}{p}}$. Similarly, if $p=\infty$, then $\|\gamma'(t)\|=\max\{d(x_1,y_1),d(x_2,y_2)\}$ for all $t$, so $d(x,y)\le\max\{d(x_1,y_1),d(x_2,y_2)\}$.

Conversely, let $\gamma:[0,1] \ra X$ be a constant-speed geodesic from $x$ to $y$, so that $\|\gamma'(t)\| = d(x,y)$ for all $t$. For $i=1,2$, let $\gamma_i$ be the projection of $\gamma$ to $X_i$. We have $d(x_i,y_i)\le\|\gamma_i\|$. If $p < \infty$, then Jensen's inequality gives 
\[
\|\gamma_i\|^p \,=\, \left( \int_0^1 \|\gamma'_i(t)\| \,dt \right)^p 
    \,\leq\, \int_0^1 \|\gamma'_i(t)\|^p \,dt,
\]
and it follows that 
\[
d(x_1,y_1)^p+d(x_2,y_2)^p \,\leq\, \int_0^1 (\|\gamma'_1(t)\|^p + \|\gamma'_2(t)\|^p)\,dt 
    \,=\, \int_0^1 \|\gamma'(t)\|^p \,dt \,=\, d(x,y)^p.
\]
If $p=\infty$, then we have
\[
d(x_i,y_i) \,\leq\, \int_0^1 \|\gamma'_i(t)\| \,dt \,\leq\, \int_0^1 \|\gamma'(t)\| \,dt \,=\, d(x,y),
\]
and hence $\max(d(x_1,y_1),d(x_2,y_2)) \leq d(x,y)$. This completes the proof.
\end{proof}

When considering geodesics in a product, we can reduce to geodesics in the factors, as shown by the following general lemma.

\begin{lem} \label{lem:geodesic_product}
Let $X_1$ and $X_2$ be geodesic metric spaces and let $p \in (1,\infty)$. Endow $X_1\times X_2$ with the $\ell^p$ metric $d\big((x_1,x_2),(y_1,y_2)\big)^p = d(x_1,y_1)^p + d(x_2,y_2)^p$. A path $\gamma\colon[0,1] \to X_1\times X_2$ is a constant-speed geodesic from $x=(x_1,x_2)$ to $y=(y_1,y_2)$ if and only if its projections $\gamma_i=\pi_{X_i}\gamma$ are constant-speed geodesics.
\end{lem}

\begin{proof}
If the $\gamma_i$ are constant-speed geodesics, then $\gamma$ is constant-speed. Moreover, $\|\gamma'(t)\|^p=\|\gamma'_1(t)\|^p+\|\gamma'_2(t)\|^p=d(x,y)^p$ for all $t$, which shows that $\|\gamma\|=d(x,y)$, so $\gamma$ is a geodesic.

For the converse, suppose that $\gamma$ is a constant-speed geodesic from $x=(x_1,x_2)$ to $y=(y_1,y_2)$, so that $\|\gamma'(t)\|^p=d(x,y)^p=d(x_1,y_1)^p+d(x_2,y_2)^p$. We know that $\int_0^1\|\gamma_1'(t)\|dt=d(x_1,y_1)$ and $\int_0^1\|\gamma_2'(t)\|=d(x_2,y_2)$. Applying Jensen's inequality, we compute
\begin{align*}
d(x_1,y_1)^p \,&=\, \left(\int_0^1\|\gamma_1'(t)\|dt\right)^p \\
    &\ge\,  \int_0^1\|\gamma_1'(t)\|^pdt \\
    &=\,    \int_0^1\left(d(x,y)^p-\|\gamma_2'(t)\|^p\right)dt \\
    &\ge\,  d(x,y)^p-\left(\int_0^1\|\gamma_2'(t)\|dt\right)^p \,=\,d(x_1,y_1)^p.
\end{align*}
Both the intervening inequalities must be equalities, which implies that the $\|\gamma_i'(t)\|$ are constant. Since $d(x,y)^p=\int_0^1\|\gamma_1'(t)\|^p+\|\gamma_2'(t)\|^pdt$, the $\gamma_i$ must be geodesics.
\end{proof}

\subsection{The zero-tension condition} \label{subsec:zero_tension}

We now specialise to the setting of CAT(0) cube complexes equipped with the $\ell^p$-metric, for some $p \in (1,\infty)$. As we are working locally, our next two lemmas consider small CAT(0) cube complexes.



\begin{lem} \label{lem:unique_geodesic_3cubes}
Let $X$ be a CAT(0) cube complex consisting of three cubes. Between any two points of $X$, there exists exactly one local $\ell^p$-geodesic.
\end{lem}

\begin{proof}
Since $X$ is a CAT(0) cube complex, we can label the three cubes $C$, $C'$, and $D$ so that $C \cap C' \subset D$. Note that any local $\ell^p$-geodesic in $X$ restricts to an affine path in each cube. 

Given $x \in C$ and $y \in C'$, any local $\ell^p$-geodesic $\gamma$ from $x$ to $y$ starts with an affine path to some $x_1 \in C \cap D$, and ends with an affine path from some $x_2 \in D \cap C'$. Furthermore, between $x_1$ and $x_2$ the path $\gamma$ does not intersect the interior of $C$ or $C'$; as a consequence, it is the affine path between $x_1$ and $x_2$ in $D$.

As a sum of strictly convex functions, the function
\begin{align*} 
(C \cap D) \times (C' \times D) &\ra \R \\
    (p_1,p_2) &\mapsto d(x,p_1)+d(p_1,p_2)+d(p_2,y)
\end{align*}
is strictly convex, so it has a unique local minimum. Hence there exists a unique local $\ell^p$-geodesic from $x$ to $y$. 

A similar but simpler argument applies if at least one of $x$ and $y$ lies in $D$.
\end{proof}

The next statement requires some notational clarifications. If $a$ and $b$ are points in a cube $C$ of dimension $d$, then by identifying $C$ with $[0,1]^d$, we can view $a$ and $b$ as vectors in $\R^d$. This allows us to perform vector operations on $a$ and $b$, and to write expressions such as $a-b$. It is also natural to write $\|a-b\|$ for the distance in $C$ from $a$ to $b$. 

Now suppose that $X=C\times Y$ is a decomposition of a CAT(0) cube complex, where $C$ is a single cube of dimension $d$. Given two points in $X$, we can consider their projections to $C$, and then consider those as vectors in $\R^d$. To simplify notation, we write expressions such as $\frac{(a-b)_C}{\|a-b\|_C}$ in place of $\frac{\pi_C(a)-\pi_C(b)}{d(\pi_C(a),\pi_C(b))}$, and so forth.

\begin{lem} \label{lem:zero_tension}
Let $C,C',D$ be cubes in a CAT(0) cube complex with $C\cap C'=D$. Let $x \in C\smallsetminus D$, $y \in C'\smallsetminus D$, and $z \in D$. The piecewise affine path from $x$ to $y$ via $z$ is an $\ell^p$-geodesic in $C \cup C'$ if and only if it satisfies the ``zero-tension condition'':
\[
\f{(x-z)_D}{d(x,z)} + \f{(y-z)_D}{d(y,z)}=0.
\]
\end{lem}

\begin{figure}[ht]
\includegraphics[height=2.8cm]{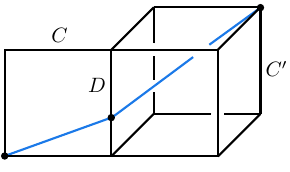}
\caption{The zero-tension condition: $\f{(x-z)_D}{d(x,z)} + \f{(y-z)_D}{d(y,z)}=0$} \label{fig:zero_tension}
\end{figure}

\begin{proof}
See Figure~\ref{fig:zero_tension}. Let $\gamma$ denote the union of the geodesic in $C$ from $x$ to $z$ with the geodesic in $C'$ from $z$ to $y$. Parametrise $\gamma$ so that it is constant-speed. Let us write $C=D \times C_0$ and $C'=D \times C'_0$, so that $C_0$ and $C'_0$ intersect in a vertex $v$. According to Lemma~\ref{lem:lp_product} and Lemma~\ref{lem:geodesic_product}, $\gamma$ is an $\ell^p$ geodesic in $C \cup C'=D \times (C_0 \cup C'_0)$ if and only if $\pi_D\gamma$ and $\pi_0\gamma$ are constant-speed $\ell^p$ geodesics in $D$ and $C_0 \cup C'_0$, respectively.

\mk

Suppose that $\gamma$ is an $\ell^p$-geodesic, so that $d(y,z)=d(x,y)-d(x,z)$. Because the unique $\ell^p$-geodesic in $D$ from $\pi_D(x)$ to $\pi_D(y)$ is affine, the point $z$ must lie along this affine path, so $\pi_D(x)$, $z$, and $\pi_D(y)$ are collinear. Also, $\pi_D\gamma$ is constant-speed, so $\frac{(z-x)_D}{\|z-x\|_D}=\frac{(y-z)_D}{\|y-z\|_D}$. Even more, the fact that $\pi_0\gamma$ is constant-speed as well means that there exists $\lambda$ such that $\|\gamma(t)-\gamma(t')\|_D=\lambda d(\gamma(t),\gamma(t'))$ for all $t,t'$. Rearranging gives the zero-tension condition.

\mk

Conversely, suppose that the zero-tension condition holds. Because there exist scalars $\lambda,\mu$ such that $\lambda(x-z)_D=\mu(y-z)_D$, the points $\pi_D(x)$, $z$, and $\pi_D(y)$ must be collinear, so $\pi_D\gamma$ is the geodesic from $\pi_D(x)$ to $\pi_D(y)$. Also, since $z$ is the unique point in $\gamma\cap D$, the projection $\pi_0\gamma$ is piecewise linear from $\pi_0(x)$ to $\pi_0(y)$ through $v$, and thus is the unique $\ell^p$-geodesic from $\pi_0(x)$ to $\pi_0(y)$ in $C_0 \cup C'_0$. 

As the projection of a concatenation of affine paths, $\pi_D\gamma$ can only fail to be locally constant-speed at the break-point $z$. Let $t$ be such that $\gamma(t)=z$. For $s<t$, the speed of $\pi_D\gamma$ is $\frac{d(\pi_D(x),z)}{d(x,z)}$. For $s>t$, the speed of $\pi_D\gamma$ is $\frac{d(z,\pi_D(y))}{d(z,y)}$. The zero-tension condition states that these agree, so $\pi_D\gamma$ is constant-speed. Since $d^p=d_D^p+d_0^p$, this forces $\pi_0\gamma$ to be constant-speed.
\end{proof}

\subsection{The no shortcut condition} \label{subsec:no_shortcut}

Whilst Lemma~\ref{lem:zero_tension} characterises local geodesics in $C\cup C'$, it does not give information about paths that avoid the intersection of $C$ and $C'$. Generically, geodesics from $C$ to $C'$ will avoid this intersection, as ``corner cubes between $C$ and $C'$'' will provide a shortcut.

Our next goal is to generalise Lemma~\ref{lem:zero_tension} to allow for these ``corner cubes''. This makes the situation more complex, and we need to add the ``no shortcut condition'' in addition, which can be thought of as dictating which corner cubes are used. From this we shall obtain an $\ell^p$ version of \cite[Thm~5.8]{ardilaowensullivant:geodesics}, before moving on to give a more explicit description of local $\ell^p$-geodesics. 

Keeping in mind Lemmas~\ref{lem:lp_product} and~\ref{lem:geodesic_product}, the next lemma takes place in a simplified setting. Note that the assumption that no subcube of $C$ contains $\{x,v\}$ is more general than assuming that $x$ lies in the interior of $C$, as it allows $x$ to be ``on the opposite side'' of $C$. This is necessary for being able to restrict to a subpath of a path that passes through many cubes.

\begin{lem} \label{lem:shortcut}
Let $C$ and $C'$ be cubes in a CAT(0) cube complex $X$ that intersect in a vertex $v$. Let $x\in C$ and $y\in C'$ be points such that no subcube of $C$ contains $\{x,v\}$ and no subcube of $C'$ contains $\{y,v\}$. Fix any decompositions $C=A_1 \times A_2$ and $C'=B_1 \times B_2$ such that $B_1 \times A_2\subset X$. The piecewise linear path $\gamma$ from $x$ to $y$ through $v$ is a local $\ell^p$-geodesic in $C \cup (B_1 \times A_2) \cup C'$ if and only if
\[ 
\f{\|x-v\|_{A_1}}{\|y-v\|_{B_1}} \geq \f{\|x-v\|_{A_2}}{\|y-v\|_{B_2}}.
\]
\end{lem}

\begin{figure}[ht]
\centering
\begin{tikzpicture}
\def \p {0.05}
\def \op {0.3}
\def \gris {blue}
\draw[fill] (-2,1.5) circle (\p) node(x) {};
\draw[fill,blue] (-2,0) circle (\p) node(x1) {};
\draw[fill,blue] (0,1.5) circle (\p) node(x2) {};
\draw[fill] (0,0) circle (\p) node(v) {};
\draw[fill] (1,-2.5) circle (\p) node(y) {};
\draw[fill,blue] (1,0) circle (\p) node(y1) {};
\draw[fill,blue] (0,-2.5) circle (\p) node(y2) {};

\draw (-3,0) -- (3,0) -- (3,3) -- (-3,3) -- (-3,0);
\draw (0,0) -- (0,3);
\draw (0,0) -- (0,-3) -- (3,-3) -- (3,0);
\draw (x.center) -- (v.center) -- (y.center);
\draw[blue,dashed] (x2.center) -- (x.center) -- (x1.center);
\draw[blue,dashed] (y2.center) -- (y.center) -- (y1.center);

\node at ([xshift=-0.5cm]x) {$x$};
\node at ([yshift=-0.5cm]x1) {\color{blue} $x_{A_1}$};
\node at ([xshift=0.5cm]x2) {\color{blue} $x_{A_2}$};
\node at ([xshift=0.5cm]y) {$y$};
\node at ([yshift=0.5cm]y1) {\color{blue} $y_{B_1}$};
\node at ([xshift=-0.5cm]y2) {\color{blue} $y_{B_2}$};
\node at ([xshift=-0.3cm,yshift=-0.3cm]v) {$v$};

\node at (-2.7,2.7) {$C$};
\node at (2.7,-2.7) {$C'$};
\node at (2.2,2.7) {$B_1 \times A_2$};
\node at (-1.5,3.3) {$A_1$};
\node at (1.5,-3.3) {$B_1$};
\node at (-3.5,1.5) {$A_2$};
\node at (3.5,-1.5) {$B_2$};

\end{tikzpicture}
\caption{The no shortcut condition: $\f{\|x-v\|_{A_1}}{\|y-v\|_{B_1}} \geq \f{\|x-v\|_{A_2}}{\|y-v\|_{B_2}}$}
\label{fig:no_shortcut}
\end{figure}
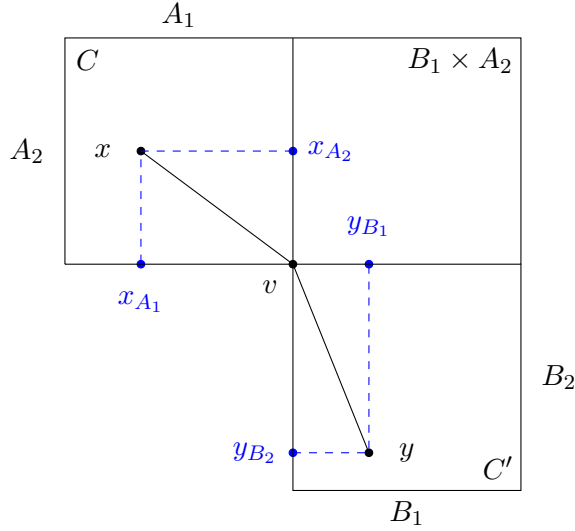

\begin{proof}
See Figure~\ref{fig:no_shortcut}. If $z\in C \cup (B_1 \times A_2) \cup C'$ and $t$ is a constant, then we write $tz$ to mean the point along the affine geodesic $[v,z]$ from $v$ to $z$ with $d(v,tz)=t$. With $v$ thus set as the basepoint, we shall simplify notation by writing $\|z\|=\|z-v\|$ where it will not cause confusion.

Let $s=\min\{d(x,v),d(y,v)\}$. The path $\gamma$ is locally geodesic if and only if its restriction from $sx$ to $sy$ is. Thus, perhaps after extending the affine segments $[v,x]$ and $[v,y]$ inside $C$ and $C'$, we may assume that $d(x,v)=d(y,v)=1$.

\mk

For points $a\in A_2$, $b\in B_1$ with $\|a\|=\|x-v\|_{A_2}$ and $\|b\|=\|y-v\|_{B_1}$, and for constants $t,\alpha\ge0$, consider the concatenation of the affine paths $[x,t\alpha a]$, $[t\alpha a,tb]$, $[tb,y]$. Let us write $f_\alpha(t)$ for the length of this path, which is given by 
\[
f_\alpha(t) \,=\, \|x-t\alpha a\| + t\|\alpha a-b\| + \|tb-y\|.
\]

When $t=0$, this concatenation degenerates to $\gamma$, so $\gamma$ is a local geodesic if and only if $f_\alpha'(0)\ge0$ for all choices of $\alpha$, $a$, and $b$. In order to compute $f_\alpha'$, let us write $C=\prod_{i=1}^nC_i$ and $C'=\prod_{i=1}^{n'}C'_i$, where the $C_i$ and $C'_i$ are unit intervals. We can then write $x=(x_1,\dots,x_n)$, $a=(a_1,\dots,a_n)$, etc. Note that $a_i=0$ if $C_i$ is not a factor of $A_2$, and $b_i=0$ if $C'_i$ is not a factor of $B_1$. By reordering, we may assume that $A_2=\prod_{i=1}^mC_i$ and $B_1=\prod_{i=1}^{m'}C'_i$. By definition, $\|x-t\alpha a\|=\left(\sum_{i=1}^n|x_i-t\alpha a_i|^p\right)^\f1p$, and so forth, and we therefore compute
\[
f_\alpha'(t) \,=\, \f{\sum_{i=1}^n -\alpha a_i|x_i-t\alpha a_i|^{p-1}\sgn(x_i-t\alpha a_i)}
        {\|x-t\alpha a\|^{p-1}}
    \,+\, \|\alpha a-b\|
    \,+\, \f{\sum_{i=1}^{n'} b_i|tb_i-y_i|^{p-1}\sgn(tb_i-y_i)}{\|tb-y\|^{p-1}}.
\]

Since $\{x,v\}$ is not contained in any subcube of $C$, every $x_i$ is positive, and similarly every $y_i$ is positive. Moreover, $A_2$ and $B_1$ are factors of a cube in $Y$, so $\alpha a$ and $b$ are orthogonal. We therefore get that 
\begin{align*}
f_\alpha'(0)   \,&=\, -\f{\sum_{i=1}^m \alpha a_ix_i^{p-1}}{d(x,v)^{p-1}} \,+\, \|\alpha a-b\|
            \,-\, \f{\sum_{i=1}^{m'} b_iy_i^{p-1}}{d(y,v)^{p-1}} \\
        &=\, -\alpha\sum_{i=1}^m a_ix_i^{p-1} \,+\, (\|\alpha a-v\|^p+\|b-v\|^p)^\f1p 
            \,-\, \sum_{i=1}^{m'} b_iy_i^{p-1} \\
        &=\, -\alpha\sum_{i=1}^m a_ix_i^{p-1} \,+\, (\alpha^p\|x-v\|_{A_2}^p + \|y-v\|_{B_1}^p)^\f1p 
            \,-\, \sum_{i=1}^{m'} b_iy_i^{p-1}.
\end{align*}
For any fixed $\alpha$, the Lagrange multiplier theorem shows that this value of $f_\alpha'(0)$ is minimised by taking $a$ and $b$ to be scalar multiples of $\pi_{A_2}(x)$ and $\pi_{B_1}(y)$, respectively. Since $\|a\|=\|x-v\|_{A_2}$ and $\|b\|=\|y-v\|_{B_1}$, that means that taking $a=\pi_{A_2}(x)$ and $b=\pi_{B_1}(y)$ minimises $f_\alpha'(0)$.

\mk

In this situation, we have that $f_\alpha'(0)= \|\alpha a-b\| - \alpha\|a\|^p - \|b\|^p$. We are reduced to dealing with the single parameter $\alpha$, so write $g(\alpha)=f_\alpha'(0)$. We have
\[
g'(\alpha) = \f{\alpha^{p-1}\|a\|^p}{\|\alpha a-b\|^{p-1}} -\|a\|^p,
\]
and so $g$ is extremised when 
\begin{align}
\alpha^p=\|\alpha a-b\|^p=\alpha^p\|a\|^p+\|b\|^p.  \label{eqn:alphaimplicit}
\end{align}
From this, we obtain
\begin{align}
\alpha^p = \f{\|b\|^p}{1-\|a\|^p}.    \label{eqn:alphaexplicit}
\end{align} 
Because all $x_i$ and $y_i$ are positive, the fact that $A_2$ and $B_1$ are proper factors of $C$ and $C'$, respectively, means that $\|x-v\|_{A_2},\|y-v\|_{B_1}\in(0,1)$. As $a=\pi_{A_2}(x)$ and $b=\pi_{B_1}(y)$, we see that the quantity in~\eqref{eqn:alphaexplicit} is positive and finite. It is straightforward to see that the unique positive value of $\alpha$ satisfying \eqref{eqn:alphaexplicit} is the minimum of $g:[0,\infty)\to\R$. At this value, we can use \eqref{eqn:alphaimplicit} and \eqref{eqn:alphaexplicit} to see that 
\begin{align*}
f_\alpha'(0) &= \|\alpha a-b\|-\alpha\|a\|^p-\|b\|^p \\
        &= \alpha - \alpha\|a\|^p - \|b\|^p \\
        &= \alpha -\alpha\|a\|^p - \alpha^p(1-\|a\|^p) \\
        &= (\|a\|^p-1)(\alpha^p-\alpha),
\end{align*}
which is nonnegative if and only if $\alpha\le1$.

\mk

We have shown that $\gamma$ is a local geodesic if and only if $\f{\|y-v\|_{B_1}^p}{1-\|x-v\|_{A_2}^p}\le1$, or equivalently $\|x-v\|_{A_2}^p+\|y-v\|_{B_1}^p\le1$. In view of the fact that $\|x-v\|^p=\|y-v\|^p=1$, this is itself equivalent to the conjunction of the statements $\|y-v\|_{B_1}^p\le\|x-v\|_{A_1}^p$ and $\|x-v\|_{A_2}^p\le\|y-v\|_{B_2}^p$, and this conjunction can be rewritten as
\begin{align*}
\f{\|x-v\|_{A_1}^p}{\|y-v\|_{B_1}^p}\ge1\ge\f{\|x-v\|_{A_2}^p}{\|y-v\|_{B_2}^p}.
\end{align*}
Returning to our original scaled copies of $x$ and $y$, this is equivalent to the inequality in the statement of the lemma.
\end{proof}

We now remove from Lemma~\ref{lem:shortcut} the assumption that the intersection of $C$ and $C'$ is a single vertex. We need to additionally incorporate the condition from Lemma~\ref{lem:zero_tension}.

\begin{lem} \label{lem:shortcut_cube}
Let $C$ and $C'$ be cubes in a CAT(0) cube complex $X$ such that $C\cap C'=D$ for some cube $D$. Let $x \in C$ and $y \in C'$ be points such that no subcube of $C$ contains $\{x,z\}$ for any $z\in D$, and no subcube of $C'$ contains $\{y,z\}$ for any $z\in D$. Fix any decompositions $C=D \times A_1 \times A_2$ and $C'=D \times B_1 \times B_2$ such that $D \times B_1 \times A_2$ belongs to $X$.
The piecewise linear path $\gamma$ from $x$ to $y$ through $z \in D$ is a local $\ell^p$-geodesic in $C \cup (D \times B_1 \times A_2) \cup C'$ if and only if
\[ 
\f{\|x-z\|_{A_1}}{\|y-z\|_{B_1}} \geq \f{\|x-z\|_{A_2}}{\|y-z\|_{B_2}} 
    \;\text{ and }\; \f{(x-z)_D}{d(x,z)} + \f{(y-z)_D}{d(y,z)} = 0.
\]
\end{lem}

\begin{proof}
The CAT(0) cube complex $C \cup (D\times B_1 \times A_2) \cup C'$ decomposes as a product $D \times Y$, where $Y=((A_1 \times A_2) \cup (B_1 \times A_2) \cup (B_1 \times B_2))$. According to Lemma~\ref{lem:lp_product} and Lemma~\ref{lem:geodesic_product}, $\gamma$ is a local $\ell^p$-geodesic if and only if its projections $\pi_D\gamma$ and $\pi_Y\gamma$ are local $\ell^p$-geodesics.

Lemma~\ref{lem:shortcut} states that $\pi_Y\gamma$ is a local $\ell^p$-geodesic if and only if $\f{\|x-z\|_{A_1}}{\|y-z\|_{B_1}} \geq \f{\|x-z\|_{A_2}}{\|y-z\|_{B_2}}$. The path $\pi_D\gamma$ is a local $\ell^p$-geodesic if and only if it is affine with constant speed. Since it is the projection of a concatenation of two affine paths, this is equivalent to the equality $\f{(z-x)_D}{d(x,z)} = \f{(y-z)_D}{d(y,z)}$.
\end{proof}

The information we can obtain from Lemmas~\ref{lem:shortcut} and~\ref{lem:shortcut_cube} is limited by the quality of the possible decompositions into two factors; in general, we need to decompose $C$ and $C'$ further. For example, consider two diagonally opposite cubes in $[0,2]^3$. Our next goal can be summed up as showing that there is always a (not-necessarily unique) ``best'' pair of decompositions to consider for a given local $\ell^p$-geodesic. First we need a simple lemma that controls the behaviour of local geodesics.

\begin{lem} \label{lem:hull}
No local $\ell^p$-geodesic $\gamma$ in a CAT(0) cube complex $X$ can cross a hyperplane twice. In particular, $\gamma$ is contained in the median hull of its endpoints.
\end{lem}

\begin{proof}
If $\gamma$ crosses a hyperplane twice, then there is some subpath $\gamma'$ as follows. The endpoints $x$ and $y$ of $\gamma'$ are contained in cubes $C$ and $C'$, respectively, and there is a hyperplane $h$ outside the median hull of $\{x,y\}$ such that $h$ is dual to both $C$ and $C'$, and there exists some $z\in\gamma'\cap h$. The hull of $\{x,y,z\}$ decomposes as a product $h\times I$, where $I$ is an interval. The projection of $\gamma'$ to $I$ is not affine, contradicting Lemma~\ref{lem:geodesic_product}.
\end{proof}

Lemma~\ref{lem:hull} explains the fact that we do not need to make dimension restrictions, because the median hull of two points in a CAT(0) cube complex is always finite-dimensional. Indeed, the dimension of the median hull of two vertices is bounded by the combinatorial distance between them.

The next lemma provides decompositions of cubes that intersect in a vertex, based on a given pair of points. The resulting inequalities are related to the zero-tension condition. See Figure~\ref{fig:distance_formula} below for more of an idea of what these decompositions mean.

\begin{lem} \label{lem:outside_v}
Let $C$ and $C'$ be cubes in a CAT(0) cube complex $X$ that intersect in a vertex $v$. Let $x \in C$ and $y \in C'$ be such that no subcube of $C$ contains $\{x,v\}$ and no subcube of $C'$ contains $\{y,v\}$. Let $\gamma:[0,1]\to X$ be a local $\ell^p$-geodesic from $x$ to $y$. There exist decompositions $C=\prod_{j=1}^kA_j$ and $C'=\prod_{j=1}^kB_j$ such that 
\[ 
\f{\|x-v\|_{A_1}}{\|y-v\|_{B_1}} <  \f{\|x-v\|_{A_2}}{\|y-v\|_{B_2}} 
    < \dots <  \f{\|x-v\|_{A_k}}{\|y-v\|_{B_k}},
\]
and if we let $C_j=B_1 \times \dots \times B_j \times A_{j+1} \times \dots \times A_k\subset X$ for each $j$, then $\gamma$ is the piecewise affine path joining points $x_0=x,x_1,\dots,x_{k+1}=y$, where $x_j\in C_{j-1}\cap C_j$ for $j\in\{1,\dots,k\}$. 
\end{lem}

\begin{proof}
Let $C=\prod_{i=1}^nI_i$ and $C'=\prod_{i=1}^mJ_i$ be decompositions into products of unit intervals such that $\pi_{I_i}(v)=0$ and $\pi_{J_i}(v)=0$ for all $i$. According to Lemma~\ref{lem:hull}, we may assume that $X$ is equal to the median hull of $C\cup C'$. Consequently, every cube of $X$ can be written as a product $\prod_{i\in\mathcal I}I_i\times\prod_{i\in\mathcal J}J_i$ for some $\mathcal I\subset\{1,\dots,n\}$, $\mathcal J\subset\{1,\dots,m\}$. 

Let $x_0=x$, $t_0=0$, and $C_0=C$, so that $\gamma(t_0)=x_0\in C_0$. Given $x_j\ne y$ and $t_j<1$, if $\gamma|_{[t_j,1]}$ is contained in $C'$, then set $t_{j+1}=1$, $x_{j+1}=y$. Otherwise, there exists a minimal $t_{j+1}\in(t_j,1]$ such that $x_{j+1}=\gamma(t_{j+1})$ lies in the boundary of $C_j$. In this case, there is a cube $C_{j+1}$ such that $\gamma|_{(t_{j+1},t_{j+1}+\eps)}$ lies in the interior of $C_{j+1}$ for all sufficiently small $\eps$, because $\gamma$ is piecewise affine. Note that this choice of $C_{j+1}$ ensures that if $x_{k+1}=y$, then $C_k=C'$.

In this way, we obtain a sequence of minimal cubes $C_0=C,C_1,\dots,C_k=C'$ in $X$ such that $\gamma$ is contained in their union, and $x_j\in C_{j-1}\cap C_j$ for all $j\in\{1,\dots,k\}$. Note that $k\ge1$, because $C\ne C'$, and $\gamma$ passes through $v$ if and only if $k=1$.

For $j\in\{1,\dots,k\}$, let $\mathcal I_j$ and $\mathcal J_j$ be such that $C_j=\prod_{i\in\mathcal I_j}I_i\times\prod_{i\in\mathcal J_j}J_i$. Note that $\mathcal I_0=\{1,\dots,n\}$, $\mathcal J_k=\{1,\dots,m\}$, and $\mathcal J_0=\mathcal I_k=\varnothing$. According to Lemma~\ref{lem:hull}, $\gamma$ cannot cross any hyperplane twice, so if $i\in\mathcal J_j$, then $i\in\mathcal J_{j'}$ for all $j'>j$. Similarly, if $i\not\in\mathcal I_j$, then $i\not\in\mathcal I_{j'}$ for all $j'>j$. Because of this, we can define 
\[
B_j\,=\!\prod_{i\in\mathcal J_j\smallsetminus\mathcal J_{j-1}} 
    \hspace{-3mm}J_i\; \text{ for } j>0, \text{ and }\quad
A_j\,=\!\prod_{i\in\mathcal I_j\smallsetminus\mathcal I_{j+1}}
    \hspace{-3mm}I_i\; \text{ for } j<k.
\]
By construction, these cubes have the property that 
\[
C_j=B_1\times\dots\times B_j\times A_{j+1}\times\dots\times A_k.
\]

For fixed $i\in\{1,\dots,k-1\}$, consider the restriction $\gamma_i$ of $\gamma$ between $x_{i-1}$ and $x_{i+2}$, which is contained in $C_{i-1}\cup C_i\cup C_{i+1}$. This union of three cubes can be decomposed as
\[
C_{i-1} \cup C_i \cup C_{i+1} = D_i \times \left( (A_i \times A_{i+1}) \cup (B_i \times A_{i+1}) \cup (B_i \times B_{i+1})\right),
\]
where $D_i$ is the cube $B_1 \times \dots \times B_{i-1} \times A_{i+2} \times \dots \times A_k$. According to Lemma~\ref{lem:unique_geodesic_3cubes}, $\gamma_i$ is the unique local $\ell^p$-geodesic in $C_{i-1}\cup C_i\cup C_{i+1}$ between its endpoints, and Lemma~\ref{lem:geodesic_product} implies the same for its projections to the two factors above. Because $\gamma_i$ avoids $D_i\times v$, its projection to the second factor avoids $v$. Lemma~\ref{lem:shortcut} now tells us that
\begin{align} 
\f{\|x_{i-1}-v\|_{A_i}}{\|x_{i+2}-v\|_{B_i}} \,<\, 
    \f{\|x_{i-1}-v\|_{A_{i+1}}}{\|x_{i+2}-v\|_{B_{i+1}}}. \label{eqn:tension}
\end{align}

The restriction of the path $\gamma$ between $x$ and $x_{i-1}$ lives in $C_0 \cup C_1 \cup \dots \cup C_{i-1}$, which contains $A_i \times A_{i+1}$ as a factor. According to Lemma~\ref{lem:geodesic_product}, we deduce that the projection of this subpath onto $A_i \times A_{i+1}$ is affine with constant speed. In particular, $A_i$ is a factor, so if we consider the projection to $A_i$ then we get 
\[
\f{\|x_{i-1}-v\|_{A_i}}{d(x_{i-1},v)} \,=\, \f{\|x-v\|_{A_i}}{d(x,v)}.
\]
Similarly, $A_{i+1}$ is a factor, so considering the projection to $A_{i+1}$ gives
\[
\f{\|x_{i-1}-v\|_{A_{i+1}}}{d(x_{i-1},v)} \,=\, \f{\|x-v\|_{A_{i+1}}}{d(x,v)}.
\]
In other words, there is a constant $\lambda=\f{d(x_{i-1},v)}{d(x,v)}$ such that 
\[
\|x_{i-1}-v\|_{A_i} = \lambda\|x-v\|_{A_i} \quad\text{ and }\quad 
    \|x_{i-1}-v\|_{A_{i+1}} = \lambda\|x-v\|_{A_{i+1}}.
\]
By the same reasoning, there is a constant $\mu=\f{d(x_{i+2},v)}{d(y,v)}$ such that
\[
\|x_{i+2}-v\|_{B_i} = \mu\|y-v\|_{B_i} \quad\text{ and }\quad 
    \|x_{i+2}-v\|_{B_{i+1}} = \mu\|y-v\|_{B_{i+1}}.
\]
Combining these with Equation~\eqref{eqn:tension}, we find that
\[
\f{\|x-v\|_{A_i}}{\|y-v\|_{B_i}} \,<\,  \f{\|x-v\|_{A_{i+1}}}{\|y-v\|_{B_{i+1}}}.
\]
As this holds for any $1 \leq i \leq k-1$, we conclude that 
\[ 
\f{\|x-v\|_{A_1}}{\|y-v\|_{B_1}} \,<\,  \f{\|x-v\|_{A_2}}{\|y-v\|_{B_2}} \,<\, \dots 
    \,<\,  \f{\|x-v\|_{A_k}}{\|y-v\|_{B_k}}. \qedhere
\]
\end{proof}

\subsection{Characterisation of local geodesics}

We are now in a position to give our characterisation of local $\ell^p$-geodesics in CAT(0) cube complexes. We use the following simple fact.

\begin{lem} \label{lem:amgm}
Let $a,b,c,d>0$, $p\ge1$. If $\f ab<\f cd$, then $\f{(a^p+c^p)^{\f1p}}{(b^p+d^p)^{\f1p}}<\f cd$.
\end{lem}

\begin{proof}
We have $\f db<\f ca$, so $ab(1+(\f db)^p)^{\f1p}<ab(1+(\f ca)^p)^{\f1p}$, and the result follows.
\end{proof}

\begin{thm} \label{thm:description_local_geodesics}
Let $X$ be a CAT(0) cube complex, endowed with the piecewise $\ell^p$ metric for some $p \in (1,\infty)$, and let $x,y \in X$. Let $\gamma$ be a piecewise affine path from $x$ to $y$ in $X$, with break points $x_0=x,x_1,\dots,x_{k+1}=y$. For each $i$, let $C_i$ denote the minimal cube containing $\{x_i,x_{i+1}\}$. The path $\gamma$ is a local $\ell^p$-geodesic if and only if for every $i\in\{1,\dots,k\}$ we have
\begin{itemize}
\item The ``zero-tension condition'' between $x_{i-1}$, $x_i$, and $x_{i+1}$:
\[
\f{(x_{i-1}-x_i)_{C_{i-1} \cap C_i}}{d(x_{i-1},x_i)} 
    + \f{(x_{i+1}-x_i)_{C_{i-1} \cap C_i}}{d(x_{i+1},x_i)}=0;
\]
\item The ``no shortcut condition'' between $x_{i-1}$, $x_i$, and $x_{i+1}$: for any decompositions $C_{i-1}=A_1 \times A_2$ and $C_i=B_1 \times B_2$ such that $B_1 \times A_2$ belongs to $X$, we have
\[
\f{\|x_{i-1}-x_i\|_{A_1}}{\|x_{i+1}-x_i\|_{B_1}} \geq \f{\|x_{i-1}-x_i\|_{A_2}}{\|x_{i+1}-x_i\|_{B_2}}.
\]
\end{itemize}
\end{thm}

\begin{proof}
The forward direction is given by Lemma~\ref{lem:shortcut_cube}. For the converse, suppose that both conditions hold between $x_{i-1}$, $x_i$, and $x_{i+1}$. Let $\gamma_i$ be the restriction of $\gamma$ between $x_{i-1}$ and $x_{i+1}$. We shall prove that $\gamma_i$ is locally geodesic at $x_i$. This suffices, because each component of $\gamma\smallsetminus\{x_1,\dots,x_k\}$ is an affine segment in a cube of $X$. 

Let $X_i$ be the median hull of $C_{i-1}\cup C_i$. According to Lemma~\ref{lem:hull}, it suffices to show that $\gamma_i$ is locally geodesic at $x_i$ inside $X_i$. Let us write $D_i=C_{i-1} \cap C_i$, so that $X_i=D_i \times Y_i$, where $Y_i$ is a finite-dimensional CAT(0) cube complex containing a distinguished vertex $v$ corresponding to $D_i$. 

According to Lemma~\ref{lem:lp_product} and Lemma~\ref{lem:geodesic_product}, $\gamma_i$ is locally geodesic at $x_i$ if and only if its two projections $\pi_{D_i}\gamma_i$ and $\pi_{Y_i}\gamma_i$ are constant-speed local geodesics. The path $\pi_{D_i}\gamma_i$ is a constant-speed geodesic if and only if it is affine, which is indeed the case because of the ``zero-tension condition''. Parametrising $\gamma$ so that it is constant-speed forces $\pi_{Y_i}\gamma_i$ to be constant-speed, so it suffices to show that $\pi_{Y_i}\gamma_i$ is a local geodesic at $v=\pi_{Y_i}(x_i)$. 

Let us write $C_{i-1}=D_i\times C'_{i-1}$ and $C_i=D_i\times C'_i$, where $C'_{i-1}$ and $C'_i$ are cubes of $Y_i$ meeting in $v$. Observe that $\pi_{Y_i}\gamma_i$ consists of the concatenation of the affine segment in $C'_{i-1}$ from $\pi_{C'_{i-1}}(x_{i-1})$ to $v$ with the affine segment in $C'_i$ from $v$ to $\pi_{C'_i}(x_{i+1})$. 

\mk

For a contradiction, assume that $\pi_{Y_i}\gamma_i$ is not locally geodesic at $v$. As $Y_i$ is finite-dimensional, Lemma~\ref{lem:geodesic} tells us that there exists a geodesic $\gamma'_{Y_i}$ from $x'=\pi_{Y_i}(x_{i-1})$ to $y'=\pi_{Y_i}(x_{i+1})$. This must be distinct from $\pi_{Y_i}\gamma_i$, and hence cannot pass through $v$. Applying Lemma~\ref{lem:outside_v} to $\gamma'_{Y_i}$, we find that there is some $h \geq 2$ and decompositions $C'_{i-1}=\prod_{j=1}^hA'_j$ and $C'_i =\prod_{j=1}^hB'_j$ such that, for any $j\in\{1,\dots,h-1\}$, the cube $Q_j=B'_1 \times \dots \times B'_j \times A'_{j+1} \times \dots \times A'_h$ belongs to $Y_i$, and 
\[
\f{\|x'-v\|_{A'_1}}{\|y'-v\|_{B'_1}} \,<\,  \f{\|x'-v\|_{A'_2}}{\|y'-v\|_{B'_2}} 
    \,<\, \dots \,<\,  \f{\|x'-v\|_{A'_h}}{\|y'-v\|_{B'_h}}. 
\]
Note that $(x_{i-1}-x_i)_{A'_j}=(x'-v)_{A'_j}$ and $(x_{i+1}-x_i)_{B'_j}=(y'-v)_{B'_j}$ for all $j$. We can therefore rewrite this as
\begin{align} \label{eqn:gamma'shortcut}
\f{\|x_{i-1}-x_i\|_{A'_1}}{\|x_{i+1}-x_i\|_{B'_1}} \,<\,  \f{\|x_{i-1}-x_i\|_{A'_2}}{\|x_{i+1}-x_i\|_{B'_2}} 
    \,<\, \dots \,<\,  \f{\|x_{i-1}-x_i\|_{A'_h}}{\|x_{i+1}-x_i\|_{B'_h}}. 
\end{align}

By the ``zero-tension condition'', the proportion of $d(x_{i-1},x_i)$ coming from $D_i$ is the same as the proportion of $d(x_{i+1},x_i)$ coming from $D_i$. In other words, we have
\[
\f{\|x_{i-1}-x_i\|_{D_i}}{d(x_{i-1},x_i)} \,=\, \f{\|x_{i+1}-x_i\|_{D_i}}{d(x_{i+1},x_i)}.
\]
Considering only factors of $C_{i-1}'$ and $C_i'$, we deduce that
\begin{align} \label{eqn:proportions}
\sum_{j=1}^h \f{\|x_{i-1}-x_i\|_{A'_j}}{d(x_{i-1},x_i)} \,=\, 
    \sum_{j=1}^h \f{\|x_{i+1}-x_i\|_{B'_j}}{d(x_{i+1},x_i)}.
\end{align}
One consequence of this is that we cannot have $\f{\|x_{i-1}-x_i\|_{A'_h}}{d(x_{i-1},x_i)} < \f{\|x_{i+1}-x_i\|_{B'_h}}{d(x_{i+1},x_i)}$, for then Equation~\eqref{eqn:gamma'shortcut} would tell us that $\f{\|x_{i-1}-x_i\|_{A'_j}}{d(x_{i-1},x_i)} < \f{\|x_{i+1}-x_i\|_{B'_j}}{d(x_{i+1},x_i)}$ for all $j$, contradicting Equation~\eqref{eqn:proportions}. We therefore have
\begin{align} \label{eqn:h_ratio}
\f{\|x_{i+1}-x_i\|_{B'_h}}{d(x_{i+1},x_i)} \,\le\, \f{\|x_{i-1}-x_i\|_{A'_h}}{d(x_{i-1},x_i)}.
\end{align}
By applying Lemma~\ref{lem:amgm} to Equation~\eqref{eqn:gamma'shortcut}, we see that 
\begin{align} \label{eqn:1_ratio}
 \f{\|x_{i-1}-x_i\|_{A'_1\times\dots\times A'_{h-1}}}{\|x_{i+1}-x_i\|_{B'_1\times\dots\times B'_{h-1}}}
    \,<\, \f{\|x_{i-1}-x_i\|_{A'_h}}{\|x_{i+1}-x_i\|_{B'_h}}.
\end{align}

Armed with these inequalities, let
\begin{align*}
A_1 \,=\, D\times A'_1\times A'_2\times\dots\times A'_{h-1}, &\quad A_2\,=\,A'_h, \\
B_1 \,=\, D\times B'_1\times B'_2\times\dots\times B'_{h-1}, &\quad B_2\,=\,B'_h.
\end{align*}
We have $C_{i=1}=A_1\times A_2$ and $C_i=B_1\times B_2$. Moreover, the cube $B_1\times A_2=D\times B'_1\times\dots\times B'_{h-1}\times A'_h$ belongs to $X$.

From the ``zero-tension condition'', we compute
\begin{align*}
\alpha\,&\vcentcolon=\, \f{\|x_{i-1}-x_i\|_{A_1}^p}{d(x_{i-1},x_i)^p} 
        \cdot \f{\|x_{i+1}-x_i\|_{B_2}^p}{d(x_{i+1},x_i)^p} \\
&=\, \f{\|x_{i-1}-x_i\|_{D_i}^p + \|x_{i-1}-x_i\|_{A'_1\times\dots\times A'_{h-1}}^p} {d(x_{i-1},x_i)^p} 
        \cdot \f{\|x_{i+1}-x_i\|_{B'_h}^p}{d(x_{i+1},x_i)^p} \\
&=\, \f{\|x_{i+1}-x_i\|_{D_i}^p\|x_{i+1}-x_i\|_{B'_h}^p}{d(x_{i+1},x_i)^pd(x_{i+1},x_i)^p}
    \,+\, \f{\|x_{i-1}-x_i\|_{A'_1\times\dots\times A'_{h-1}}^p\|x_{i+1}-x_i\|_{B'_h}^p}
            {d(x_{i-1},x_i)^pd(x_{i+1},x_i)^p}.
\end{align*}
Applying Equation~\eqref{eqn:h_ratio} to the first term of this expression, we find that
\[
\alpha \,\le\, \f{\|x_{i+1}-x_i\|_{D_i}^p\|x_{i-1}-x_i\|_{A'_h}^p}{d(x_{i+1},x_i)^pd(x_{i-1},x_i)^p}
    \,+\, \f{\|x_{i-1}-x_i\|_{A'_1\times\dots\times A'_{h-1}}^p\|x_{i+1}-x_i\|_{B'_h}^p}
            {d(x_{i-1},x_i)^pd(x_{i+1},x_i)^p},
\]
and applying (a rearranged) Equation~\eqref{eqn:1_ratio} to the second term gives
\begin{align*}
\alpha \,&\le\, \f{\|x_{i+1}-x_i\|_{D_i}^p\|x_{i-1}-x_i\|_{A'_h}^p}{d(x_{i-1},x_i)^pd(x_{i+1},x_i)^p}
    \,+\, \f{\|x_{i+1}-x_i\|_{B'_1\times\dots\times B'_{h-1}}^p\|x_{i-1}-x_i\|_{A'_h}^p}
            {d(x_{i-1},x_i)^pd(x_{i+1},x_i)^p} \\
    &=\, \f{\|x_{i+1}-x_i\|_{B_1}^p}{d(x_{i+1},x_i)^p}
        \cdot \f{\|x_{i-1}-x_i\|_{A_2}^p}{d(x_{i-1},x_i)^p}.
\end{align*}
From this, we conclude that 
\[
\f{\|x_{i-1}-x_i\|_{A_1}}{\|x_{i+1}-x_i\|_{B_1}} \,<\, \f{\|x_{i-1}-x_i\|_{A_2}}{\|x_{i+1}-x_i\|_{B_2}}.
\]
This contradicts the ``no shortcut condition'', which we are assuming, so $\pi_{Y_i}\gamma_i$ is locally geodesic at $v$. This completes the proof.
\end{proof}

\subsection{Unique geodesicity and distance formula}

Now that we have a characterisation of $\ell^p$-geodesics in terms of the ``zero tension'' and ``no shortcut'' conditions, we use it to show that there can only be one $\ell^p$-geodesic joining any given pair of points in the neighbourhood of a vertex of $X$. In other words, $X$ is locally uniquely $\ell^p$-geodesic. Because Lemma~\ref{lem:outside_v} shows that any local $\ell^p$-geodesic determines a collection of decompositions satisfying a certain chain of inequalities, this essentially amounts to showing that only one such decomposition can exist.

\begin{pro} \label{prop:unique_decomposition}
Let $C$ and $C'$ be cubes in a CAT(0) cube complex $X$ that intersect in a vertex $v$. Let $x \in C$ and $y \in C'$ be such that no subcube of $C$ contains $\{x,v\}$ and no subcube of $C'$ contains $\{y,v\}$. There exist unique maximal decompositions $C=\prod_{j=1}^kA_j$ and $C'=\prod_{j=1}^kB_j$ such that, for each $j\in\{1,\dots,k-1\}$, the cube $C_j=B_1 \times \dots \times B_j \times A_{j+1} \times \dots \times A_k$ belongs to $X$ and moreover 
\begin{align} \label{eqn:maximal}
\f{\|x-v\|_{A_1}}{\|y-v\|_{B_1}} \,<\, \f{\|x-v\|_{A_2}}{\|y-v\|_{B_2}} 
    \,<\, \dots \,<\, \f{\|x-v\|_{A_k}}{\|y-v\|_{B_k}}. 
\end{align}
\end{pro}

\begin{proof}
Such decompositions exist because of Lemma~\ref{lem:outside_v}, so we must show that they are unique. We first prove the statement for $p=2$. 

Given maximal decompositions $C=\prod_{j=1}^kA_j$ and $C'=\prod_{j=1}^kB_j$ such that every $C_j=B_1 \times \dots \times B_j \times A_{j+1} \times \dots \times A_k$ belongs to $X$ and such that Equation~\eqref{eqn:maximal} is satisfied for the norm $\ell^2$, consider the geodesic $\gamma$ from $x$ to $y$ in $C\cup C'\cup\bigcup_{j=1}^k C_j$. According to Theorem~\ref{thm:description_local_geodesics}, $\gamma$ is a local geodesic in $X$. But $X$ is CAT(0), so $\gamma$ is the unique global geodesic from $x$ to $y$. This shows that the decompositions are unique when $p=2$.

\mk

We now turn to the case $p \neq 2$. For simplicity, identify the vertex $v$ with $0$ in each cube in the median hull of $C\cup C'$. Let us consider the point $x'\in C=[0,1]^d$ whose $j^\mathrm{th}$ coordinate is $x'_i = x_i^{\f{p}{2}}$ for each $j$. Similarly, consider the point $y' \in C'=[0,1]^{d'}$ whose $j^\mathrm{th}$ coordinate is $y'_i = y_i^{\f{p}{2}}$ for each $j$. For each factor $A=[0,1]^r \times \{0\}^{d-r}$ of the cube $C=[0,1]^d$, we have
\[
\|x'-v\|^2_{(A,\ell^2)} \,=\, \sum_{i=1}^r{x'_i}^2 
    \,=\, \sum_{i=1}^r{x_i}^p \,=\, \|x-v\|^p_{(A,\ell^p)}{ }.
\]
Similarly, for each factor $B=[0,1]^{r'} \times \{0\}^{d'-r'}$ of the cube $C'=[0,1]^{d'}$, we have
\[
\|y'-v\|^2_{(B,\ell^2)} \,=\, \sum_{i=1}^{r'}{y'_i}^2 
    \,=\, \sum_{i=1}^{r'}{y_i}^p = \|y-v\|^p_{(B,\ell^p)}{ }.
\]
According to the case $p=2$ applied to the pair $(x',y')$, the maximal decompositions of $C$ and $C'$ satisfying Equation~\eqref{eqn:maximal} for the norm $\ell^p$ are unique.
\end{proof}

For $C$ a cube in a CAT(0) cube complex $X$, we write $\Star C$ for the union of all cubes containing $C$.

\begin{pro} \label{prop:lug}
Let $X$ be a CAT(0) cube complex, and equip $X$ with the path metric $d$ induced by giving each cube the $\ell^p$-metric. The space $(X,d)$ is locally uniquely geodesic. More specifically, every star of a vertex of $X$ is uniquely geodesic.
\end{pro}

\begin{proof}
Let $w$ be a vertex of $X$, and let $x,y\in\Star w$. Let $C$ be the minimal cube of $\Star w$ containing $\{x,w\}$, and let $C'$ be the minimal cube of $\Star w$ containing $\{y,w\}$. Let $D=C\cap C'$. By Lemma~\ref{lem:hull}, every $\ell^p$-geodesic from $x$ to $y$ is contained in the median hull $H$ of $C\cup C'$. As $H$ decomposes as $D\times Y$ for some CAT(0) cube complex $Y$, Lemma~\ref{lem:geodesic_product} states that a path $\gamma$ from $x$ to $y$ is a local geodesic if and only if $\pi_D\gamma$ and $\pi_Y\gamma$ are constant-speed geodesics. It thus suffices to show that there are unique $\ell^p$-geodesics from $\pi_D(x)$ to $\pi_D(y)$ and from $\pi_Y(x)$ to $\pi_Y(y)$. 

Since $D$ is a cube, the affine segment from $\pi_D(x)$ to $\pi_D(y)$ is the unique geodesic.

Write $C=C_Y\times D$ and $C'=C'_Y\times D$, so that $C_Y$ and $C'_Y$ intersect in a single vertex $v$ of $Y$ corresponding to $D$. The sequence of cubes passed through by a local $\ell^p$-geodesic in $Y$ from $\pi_Y(x)$ to $\pi_Y(y)$ gives rise to decompositions of $C_Y$ and $C'_Y$ satisfying a certain sequence of inequalities, as shown by Lemma~\ref{lem:outside_v}. Proposition~\ref{prop:unique_decomposition} states that the maximal such decompositions are unique. This shows that there is a fixed sequence of cubes $C_0=C_y,C_1,\dots,C_k=C'_Y$ that contains every local $\ell^p$-geodesic from $\pi_Y(x)$ to $\pi_Y(y)$. In a fixed sequence of cubes, a local $\ell^p$-geodesic is the solution of a convex optimisation problem (given by measuring the distances between break-points), so each sequence of cubes supports only one local $\ell^p$-geodesic. Thus there is a unique local $\ell^p$-geodesic in $Y$ from $\pi_Y(x)$ to $\pi_Y(y)$, which must therefore be the unique geodesic.
\end{proof}

It is worth noting that the decompositions utilised in this section give a convenient tool for computing the $\ell^p$-distance between two points of $X$. 

\begin{lem}[Distance formula] \label{lem:distance_formula}
Let $C$ and $C'$ be cubes of a CAT(0) cube complex $X$ that intersect in a vertex $v$. Let $x\in C$ and $y\in C'$ be such that $C$ is the minimal cube containing $\{x,v\}$ and $C'$ is the minimal cube containing $\{y,v\}$. Let $C=\prod_{j=1}^kA_j$ and $C'=\prod_{j=1}^kB_j$ be the decompositions provided by Proposition~\ref{prop:unique_decomposition}. The $\ell^p$-distance from $x$ to $y$ is given by
\[
\left\|\big(\|x-v\|_{A_1},\dots,\|x-v\|_{A_k}\big)
    \,+\, \big(\|y-v\|_{B_1},\dots,\|y-v\|_{B_k}\big)\right\|.
\]
\end{lem}

\begin{figure}[ht]
\includegraphics[height=5.3cm]{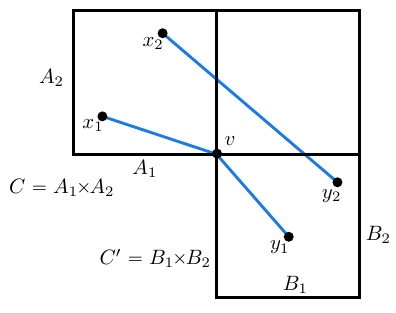}\centering
\caption{The distance from $x_1$ to $y_1$ is equal to $\|x_1-v\|_C+\|y_1-v\|_{C'}$: the decompositions are the trivial ones $C=C$, $C'=C'$. \\ 
The distance from $x_2$ to $y_2$ is the same as in the space $(A_1\vee B_1)\times(A_2\vee B_2)$: the decompositions are $C=A_1\times A_2$, $C'=B_1\times B_2$.} \label{fig:distance_formula}
\end{figure}

\begin{proof}
Let $Q_j$ be the CAT(0) cube complex obtained by taking a wedge sum of $A_j$ and $B_j$ at $v$, and consider $Q=\prod_{j=1}^kQ_j$. We can identify the median hull $H$ of $C\cup C'$ as a subcomplex of $Q$, so any path in $H$ can be viewed as a path in $Q$. Let $\gamma$ be the $\ell^p$-geodesic in $X$ from $x$ to $y$, which is contained in $H$. By Theorem~\ref{thm:description_local_geodesics}, $\gamma$ satisfies the ``zero tension'' and ``no shortcut'' conditions. Applying the converse to $Q$, we see that $\gamma$ is a geodesic of $Q$. Hence $d(x,y)=d_Q(x,y)$, which can be calculated with the formula in the statement.
\end{proof}

The product $Q$ in the proof of Lemma~\ref{lem:distance_formula} gives a convenient way of interpreting the sequence of inequalities in Equation~\eqref{eqn:maximal}: it dictates which cubes of $Q$ are met by the unique geodesic in $Q$.

\subsection{Busemann-convexity}

We now move to proving that CAT(0) cube complexes with the $\ell^p$-metric are Busemann-convex. We begin with the finite-dimensional case, which we shall use to prove the general case in Theorem~\ref{thm:busemann_cube_complex}.

\begin{lem} \label{lem:finite_dimensional_busemann}
Any finite-dimensional CAT(0) cube complex equipped with the piecewise $\ell^p$-metric is Busemann-convex.
\end{lem}

\begin{proof}
Let $X$ be a finite-dimensional CAT(0) cube complex. We prove the result by induction on $\dim X$. It is clear if $\dim X=1$, for then $X$ is a tree.

Let $v\in X$, and let $F$ denote the minimal cube containing $v$. If $F$ is a maximal cube of $X$, then $v$ has a neighbourhood isometric to a ball in $(\R^{\dim X},\ell^p)$, so is locally Busemann-convex at $v$. If $F$ is neither a maximal cube nor a vertex, then a neighbourhood of $v$ is isometric to $F\times Y$ for some CAT(0) cube complex $Y$ of dimension strictly less than $\dim X$. From the description of $\ell^p$-geodesics in products, Lemma~\ref{lem:geodesic_product}, the inductive hypothesis implies that $X$ is locally Busemann-convex at $v$.

\mk

In the remaining case, $F=v$ is a vertex of $X$. Let $U$ be the open star of $v$, which is uniquely geodesic by Proposition~\ref{prop:lug} and convex by Lemma~\ref{lem:hull}. We aim to show that for any $x,y,y'\in U$ we have $d(\sigma_{xy}(\f12),\sigma_{xy'}(\f12))\le\f12d(y,y')$, for then we can conclude by applying Lemma~\ref{lem:busemann_sufficient}.

\mk

Let $(y_s)_{s\in[0,1]}$ denote the geodesic from $y=y_0$ to $y'=y_1$. Let us first suppose that $v$ does not lie on any $\sigma_{xy_s}|_{(0,1)}$ with $s\in(0,1)$. In this case, Lemma~\ref{lem:locally_in_star} below shows that there exists $\delta>0$ such that the following holds for any $r,r'\in(0,1)$ with $|r-r'|<\delta$: there is some $\eps>0$ such that, for any $t\in(\eps,1-\eps)$, the geodesics $\sigma_{xy_r}|_{(t-\eps,t+\eps)}$ and $\sigma_{xy_{r'}}|_{(t-\eps,t+\eps)}$ are contained in the star of some cube of dimension at least 1. As in the previous case above, we deduce that the map $t\mapsto d(\sigma_{xy_r}(t),\sigma_{xy_{r'}}(t))$ is convex on the interval $(t-\eps,t+\eps)$. Since these intervals cover $(0,1)$, the map $t\mapsto d(\sigma_{xy_r}(t),\sigma_{xy_{r'}}(t))$ is convex on $[0,1]$. In particular, $d(\sigma_{xy_r}(t),\sigma_{xy_{r'}}(t))\le td(y_r,y_{r'})$ for all $t\in[0,1]$. By the triangle inequality, we deduce from this that $d(\sigma_{xy}(t),\sigma_{xy'}(t)) \le td(y,y')$ for all $t\in[0,1]$, and in particular for $t=\f12$. 

\mk

Now suppose that there is some $s\in(0,1)$ such that $v\in\sigma_{xy_s}|_{(0,1)}$. Let $s^-$ be the minimal such $s$, and let $s^+$ be the maximal such $s$, and let $y^\pm=y_{s^\pm}$. From the above case, we know that $d(\sigma_{xy}(\f12),\sigma_{xy^-}(\f12)) \le \f12d(y,y^-)$ and $d(\sigma_{xy'}(\f12),\sigma_{xy^+}(\f12)) \le \f12d(y',y^+)$. If we can show that $d(\sigma_{xy^-}(\f12),\sigma_{xy^+}(\f12)) \le \f12d(y^-,y^+)$, then the proof of the lemma will be complete.

Note that $\sigma_{xy^-}$ and $\sigma_{xy^+}$ both have $\sigma_{xv}$ as an initial subsegment, and the terminal subsegments $\sigma_{vy^-}$ and $\sigma_{vy^+}$ have the property that $v$ does not lie in $\sigma_{vz}|_{(0,1)}$ for any $z\in\sigma_{y^-y^+}$. Let us assume that ${d(x,y^-)} \ge d(x,y^+)$, for the argument is the same if the reverse holds. Let $m^-=\sigma_{xy^-}(\f12)$ and $m^+=\sigma_{xy^+}(\f12)$.

If $d(x,v)\le\f12d(x,y^-)$, then let $r=\f{d(v,m^-)}{d(v,y^-)}=\f{d(x,y^-)-2d(x,v)}{2d(v,y^-)}\ge0$. Note that $m^-=\sigma_{vy^-}(r)$, and consider the point $p=\sigma_{vy^+}(r)$. Since neither $\sigma_{vy^-}|_{(0,1)}$ nor $\sigma_{vy^+}|_{(0,1)}$ contains $v$, we know from above that $d(m^-,p)\le rd(y^-,y^+)$. We can also compute 
\begin{align*}
d(p,&m^+) \,=\, rd(v,y^+)+d(x,v)-d(x,m^+) \\
    &=\, r(d(x,y^+)-d(x,v))+d(x,v)-\f12d(x,y^+) \\
    &=\, \f1{2d(v,y^-)}\bigg(\Big(d(x,y^-)-2d(x,v)\Big)\Big(d(x,y^+)-d(x,v)\Big)+2d(v,y^-)\Big(d(x,v)-\frac12d(x,y^+)\Big)\bigg) \\
    &=\, \f1{2d(v,y^-)}\bigg(\Big(d(x,y^-)-2d(x,v)\Big)\Big(d(x,y^+)-d(x,v)\Big)+\Big(d(x,y^-)-d(x,v)\Big)\Big(2d(x,v)-d(x,y^+)\Big)\bigg) \\
    &=\, \f{d(x,v)}{2d(v,y^-)}(d(x,y^-)-d(x,y^+)) \\
    &\le\, \f{d(x,v)}{2d(v,y^-)}d(y^-,y^+).
\end{align*}
Combining these estimates yields 
\begin{align*}
d(m^-,m^+) \,&\le\, d(m^-,p)+d(p,m^+) \\
    &\le\, \left(r+\f{2d(x,v)}{d(v,y^-)}\right)d(y^-,y^+) \\
    &=\, \f12d(y^-,y^+). 
\end{align*}

Otherwise, the geodesic $\sigma_{m^+m^-}$ is a subgeodesic of $\sigma_{xv}$. In this case, only the distances of $y^-$ and $y^+$ from $v$ affect $d(m^-,m^+)$, not the actual positions in $X$. We may therefore assume that $y^+\in\sigma_{vy^-}$, for this minimises $d(y^-,y^+)$ subject to fixed values of $d(v,y^-)$ and $d(v,y^+)$. But in this case the situation is 1-dimensional, so we are done.
\end{proof}

\begin{lem} \label{lem:locally_in_star}
Let $v$ be a vertex of a CAT(0) cube complex $X$ equipped with the piecewise $\ell^p$-metric. Let $(y_s)_{s\in[0,1]}$ be a geodesic in $\Star v$, and let $x\in\Star v$ be such that $v$ does not lie in the interior of any $\sigma_{xy_s}$ with $s\in(0,1)$. There exists $\delta>0$ such that for any $r,r'\in(0,1)$ with $|r-r'|<\delta$, there is some $\eps>0$ such that for every $t\in(\eps,1-\eps)$, the geodesics $\sigma_{xy_r}|_{(t-\eps,t+\eps)}$ and $\sigma_{xy_{r'}}|_{(t-\eps,t+\eps)}$ lie in a union of at most two cubes.
\end{lem}

\begin{proof}
Let $C_0$, $C_1$, and $C_x$ be the minimal cubes containing $y_0$, $y_1$, and $x$, respectively. By Lemma~\ref{lem:hull}, the geodesic $(y_s)$ is contained in the median hull of $C_0$ and $C_1$, and hence $\bigcup_{s\in[0,1]}\sigma_{xy_s}$ is contained in the median hull $H$ of $\{C_x,C_0,C_1\}$. Note that $H$ is a finite CAT(0) cube complex; let $n$ be the number of cubes of $H$. Lemma~\ref{lem:hull} also states that no $\sigma_{xy_s}$ can cross a hyperplane twice, so each $\sigma_{xy_s}$ can change cubes at most $n$ times.

It follows that by taking $\delta$ small enough, we can ensure that if $|r-r'|<\delta$, then for every $t\in(0,1)$, the median hull of $\sigma_{xy_r}(t)$ and $\sigma_{xy_{r'}}(t)$ is a union of at most two cubes. By decreasing $\delta$ further, we can make this hold in a small neighbourhood of $t$ for every~$t$.
\end{proof}

\begin{thm} \label{thm:busemann_cube_complex}
Let $p\in(1,\infty)$. Every CAT(0) cube complex $X$ with the piecewise $\ell^p$-metric is Busemann-convex and uniformly convex. If moreover $p \geq 2$, then $X$ is also uniformly smooth and strongly bolic.\end{thm}

\begin{proof}
According to Lemma~\ref{lem:busemann_sufficient}, it suffices to show that every point $v$ of $X$ has a convex, uniquely geodesic neighbourhood $U$ such that triples in $U$ satisfy a certain inequality. Given $v$, let $w$ be a vertex of the minimal cube containing $v$, and let $U$ be the open star of $w$. Lemma~\ref{lem:hull} tells us that $U$ is convex, and Proposition~\ref{prop:lug} tells us that it is uniquely geodesic. It remains to show that every triple in $U$ satisfies the inequality.

Let $x,y,y'\in U$. Let $C_x$, $C_y$, $C_{y'}$ be the minimal cubes containing $x$, $y$, and $y'$, respectively. By Lemma~\ref{lem:hull}, both $\sigma_{xy}$ and $\sigma_{xy'}$ are contained in the median hull $H$ of $C_x\cup C_y\cup C_{y'}$, which is a finite-dimensional CAT(0) subcomplex. By Lemma~\ref{lem:finite_dimensional_busemann}, $H$ is Busemann-convex, so the desired inequality is satisfied.

\medskip

Let $p'=\max(p,2)$. According to Lemma~\ref{lem:lp_convex}, there exists $k>0$ such that $(\R^n,\ell^p)$ is $(p',k)$-uniformly convex for all $n \in \N$. We will show that $X$ is $(p',k)$-uniformly convex.

Fix $x,y,z \in X$. By Lemma~\ref{lem:hull}, $x,y,z$ and $\sigma_{yz}$ are contained in a finite-dimensional CAT(0) convex subcomplex $H$. Since $H$ is finite-dimensional, according to Theorem~\ref{thm:convex_bicombing_implies_bolicity}, we deduce that $H$ is $(p',k)$-uniformly convex. Since $H$ is convex, we deduce that $d(x,\sigma_{yz}(\f12))^{p'} \,\leq\, \f{1}{2}d(x,y)^{p'} + \f{1}{2}d(x,z)^{p'} - kd(y,z)^{p'}$, hence $X$ is $(p',k)$-uniformly convex.

\medskip

Assume now that $p \geq 2$, we will use the same idea to prove that $X$ is uniformly smooth. According to Lemma~\ref{lem:lp_smooth}, there exists $C>0$ such that $(\R^n,\ell^p)$ is $(2,C)$-uniformly smooth for all $n \in \N$. We will show that $X$ is $(2,C)$-uniformly smooth.

Fix $r,R>0$ with $R \geq 2r$, and consider $x,y,z \in X$ with $d(x,y) \leq r$ and $d(y,z) \geq R$. By Lemma~\ref{lem:hull}, $x,y,z$ and $\sigma_{yz}$ are contained in a finite-dimensional CAT(0) convex subcomplex $H$. Since $H$ is finite-dimensional, according to Theorem~\ref{thm:convex_bicombing_implies_bolicity}, we deduce that $H$ is $(2,C)$-uniformly smooth. Since $H$ is convex, we deduce that $d(x,\sigma_{yz}(\f12)) \leq d(x,z) - \f12d(y,z) + \f{Cr^2}{R}$, hence $X$ is $(2,C)$-uniformly smooth.

In particular, if $p \geq 2$, $X$ is strongly bolic.
\end{proof}

Note that when $p<2$, it is probably true that $X$ is strongly bolic, but it will not be $(2,C)$-uniformly smooth for some constant $C>0$ in general, so we cannot apply Proposition~\ref{prop:smooth_local_to_global}.

\subsection{The cases \texorpdfstring{$p=1$}{p=1} and \texorpdfstring{$p=\infty$}{p=∞}}

Because Theorem~\ref{thm:busemann_cube_complex} holds for all values of $p\in(1,\infty)$, we can consider the limiting system of paths as we take $p\to1$ or $p\to\infty$. 

\bthm \label{thm:cube_complex_unique_bicombing_allp}
Let $X$ be a CAT(0) cube complex, endowed with the piecewise $\ell^p$-metric~$d^p$, for some $p \in [1,\infty]$. The metric space $(X,d^p)$ is CUB, i.e. it admits a unique convex geodesic bicombing $\sigma^p$. Furthermore, the following map is continuous:
\begin{align*}
X\times X\times[0,1]\times[1,\infty] \,&\longrightarrow\, X \\
(x,y,t,p) \,&\longmapsto\, \sigma^p_{xy}(t).
\end{align*} 
\ethm

\bp
For $p \in (1,\infty)$, according to Theorem~\ref{thm:description_local_geodesics}, the unique geodesic bicombing on $(X,d^p)$ is convex. The zero-tension and no-shortcut conditions (see Theorem~\ref{thm:description_local_geodesics}) imply that the map $p \ra \sigma^p$ is continuous on $(1,\infty)$.

\mk

Fix $q \in \{1,\infty\}$ and fix $x,y \in X$. Let $C_x,C_y \subset X$ denote the minimal cubes containing $x,y$, and assume that $C_x$ and $C_y$ intersect. Let us denote by $H$ the median hull of $C_x \cup C_y$: it is a finite CAT(0) cube subcomplex of $X$.

The space $H$ is compact, so we may consider a limiting $d^q$-geodesic path $\sigma^q(x,y)$ from $x$ to $y$ as an accumulation point of $d^p$-geodesic paths $\sigma^p_{xy}$ as $p \ra q$. For each $p \in (1,\infty)$, the bicombing $\sigma^p$ is convex, hence for each $z \in H$, the function $t \mapsto d(\sigma^q_{xy}(t),z)$ is convex: the path $\sigma^q_{xy}$ is called \emph{straight} in $H$. According to \cite[Prop.~4.3]{descombeslang:convex}, if $H$ has finite combinatorial dimension, then $H$ has at most one straight geodesic between any pair of points. We shall prove that $H$ has finite combinatorial dimension.

\mk

When $q=\infty$, the metric space $(H,d^\infty)$ is injective \cite{miesch:injective}, so its combinatorial dimension is at most the dimension of $H$, which is finite.

When $q=1$, the CAT(0) cube complex $(H,d^1)$ is an isometric subcomplex of the cube complex $\ov{H}=D \times C'_x \times C'_y$ with the $\ell^1$ metric, where $D=C_x \cap C_y$, $C_x = D \times C'_x$ and $C_y = D \times C'_y$. Since $\ov{H}$ is a cube, we see that it is an isometric subspace of $(\R^N,\ell^1)$, where $N=\dim(C_x)+\dim(C_y)$. 
 Let $S=\{\eps \in \{\pm\}^N \st \eps_1=+\}$. Note that the map $x \in \R^N \mapsto (\eps_1x_1+ \dots +\eps_Nx_N)_{\eps \in S} \in \R^S$ is an isometric embedding from $(\R^N,\ell^1)$ into $(\R^S,\ell^\infty)$. We deduce that $(H,d^1)$ has finite combinatorial dimension, bounded above by $|S|=2^{N-1}$, see also~\cite{herrlich:hyperconvex}.

\mk

Applying~\cite[Prop.~4.3]{descombeslang:convex} to $H$, we deduce that the sequence of paths $\sigma^p_{xy}$ converges to the unique straight path $\sigma^q_{xy}$ in $H$ as $p \ra q$.

This defines a local bicombing $\sigma^q$ on $X$, which is easily seen to be consistent, reversible, and convex. According to the end of the proof of Theorem~\ref{thm:gluing_busemann_convex}, the local convex bicombing $\sigma^q$ is also geodesic. Theorem~\ref{thm:cartan_hadamard} now tells us that $(X,d^q)$ has a unique convex, consistent, geodesic bicombing that restricts to $\sigma^q$. To show that this global bicombing is the unique convex geodesic bicombing, it suffices to show that $\sigma^q$ is the unique local convex, consistent bicombing on $(X,d^q)$. Since we have shown that, for $x,y$ in adjacent minimal cubes $C_x,C_y$, the path $\sigma^q_{xy}$ is the unique straight geodesic in the median hull $H$ of $C_x\cup C_y$, it suffices to show that every straight geodesic from $x$ to $y$ is contained in $H$.

\mk

In the case $q=1$ something stronger is true: every $d^1$ geodesic in $X$ between points of $H$ is contained in $H$.

Now consider $q=\infty$ and assume that some straight geodesic $\gamma$ in $(X,d^\infty)$ between $x$ and $y$ is not contained in $H$. As in the proof of Lemma~\ref{lem:hull}, we may assume (up to passing to a smaller interval) that there exists a cube subcomplex $Y=C \cup C'$ of $X$, with $C \cap C'=D$ equal to an edge $D=[0,1]$, and a straight geodesic $\gamma:[0,1] \ra Y$ whose projection onto $D$ is not affine. Without loss of generality (up to passing to a smaller interval again), we may assume that $x=\gamma(0) \in C$, $y=\gamma(1) \in C'$ and $z=\gamma(\f{1}{2}) \in D$. Up to the choice of parametrisation of $D=[0,1]$, we have $z_D > \f{x_D+y_D}{2}$. This contradicts the convexity of the distance to $0 \in D$.

We have shown that every straight geodesic in $(X,d^q)$ between $x$ and $y$ is contained in $H$. Hence $\sigma^q$ is unique, which completes the proof.
\ep

Note that, for the limiting cases $p=1$ and $p=\infty$, the zero-tension and no-shortcut conditions (see Theorem~\ref{thm:description_local_geodesics}) hold, but they are not sufficient to characterise the unique convex bicombing.

\begin{exe}
Consider the CAT(0) cube complex $X$ depicted in Figure~\ref{fig:zero_tension}. Identify the cube $D=C\cap C'$ with $[0,1]$, where $x_D=0$ and $y_D=1$. As $p\to\infty$, the point $z_p=D\cap\sigma^p_{xy}$ converges to the midpoint $\frac12\in D$. On the other hand, $z_p$ converges to $\frac13\in D$ as $p\to1$.
\end{exe}

In the case $p=\infty$, the space $(X,d^\infty)$ in injective, so in the case where $X$ is locally finite-dimensional, uniqueness of $\sigma^\infty$ is given by \cite[Thm~1.2]{descombeslang:convex}. For $p=1$, it seems that rather less is known. Although it is very natural, the bicombing $\sigma^1$ differs from most $\ell^1$-bicombings usually considered on CAT(0) cube complexes, for instance because there are arbitrarily long $\sigma^1$ paths that are disjoint from the 0--skeleton. It could possibly be interesting to know more about $\sigma^1$. For instance, the construction of a nice bicombing on CAT(0) cube complexes is an important part of the proof of \emph{semihyperbolicity} in \cite[Thm~5.1]{durhamminskysisto:stable}.

\bibliographystyle{alpha}
\bibliography{bibtex}

\newcommand{\etalchar}[1]{$^{#1}$}
\begin{thebibliography}{CCG{\etalchar{+}}20}

\bibitem[AB90]{alexanderbishop:hadamard}
Stephanie~B. Alexander and Richard~L. Bishop.
\newblock The {H}adamard-{C}artan theorem in locally convex metric spaces.
\newblock {\em Enseign. Math. (2)}, 36(3-4):309--320, 1990.

\bibitem[AL17]{alvarezlafforgue:actions}
Aur\'{e}lien Alvarez and Vincent Lafforgue.
\newblock Actions affines isom\'{e}triques propres des groupes hyperboliques
  sur des espaces {$\ell^p$}.
\newblock {\em Expo. Math.}, 35(1):103--118, 2017.

\bibitem[AOS12]{ardilaowensullivant:geodesics}
Federico Ardila, Megan Owen, and Seth Sullivant.
\newblock Geodesics in {$\rm CAT(0)$} cubical complexes.
\newblock {\em Adv. in Appl. Math.}, 48(1):142--163, 2012.

\bibitem[Bas24]{basso:extending}
Giuliano Basso.
\newblock Extending and improving conical bicombings.
\newblock {\em Enseign. Math.}, 70(1-2):165--196, 2024.

\bibitem[BBF21]{bestvinabrombergfujiwara:proper}
Mladen Bestvina, Ken Bromberg, and Koji Fujiwara.
\newblock Proper actions on finite products of quasi-trees.
\newblock {\em Ann. H. Lebesgue}, 4:685--709, 2021.

\bibitem[BCL94]{ballcarlenlieb:sharp}
Keith Ball, Eric~A. Carlen, and Elliott~H. Lieb.
\newblock Sharp uniform convexity and smoothness inequalities for trace norms.
\newblock {\em Invent. Math.}, 115(3):463--482, 1994.

\bibitem[BH99]{bridsonhaefliger:metric}
Martin~R. Bridson and Andr\'{e} Haefliger.
\newblock {\em Metric spaces of non-positive curvature}, volume 319 of {\em
  Grundlehren der Mathematischen Wissenschaften}.
\newblock Springer-Verlag, Berlin, 1999.

\bibitem[BI13]{buragoivanov:polyhedral}
Dmitri Burago and Sergei Ivanov.
\newblock Polyhedral {F}insler spaces with locally unique geodesics.
\newblock {\em Adv. Math.}, 247:343--355, 2013.

\bibitem[BK02]{bucherkarlsson:ondefinition}
Michelle Bucher and Anders Karlsson.
\newblock On the definition of bolic spaces.
\newblock {\em Expo. Math.}, 20(3):269--277, 2002.

\bibitem[BM11]{behrstockminsky:centroids}
Jason Behrstock and Yair~N. Minsky.
\newblock Centroids and the rapid decay property in mapping class groups.
\newblock {\em J. Lond. Math. Soc. (2)}, 84(3):765--784, 2011.

\bibitem[Bri91]{bridson:geodesics}
Martin~R. Bridson.
\newblock Geodesics and curvature in metric simplicial complexes.
\newblock In {\em Group theory from a geometrical viewpoint ({T}rieste, 1990)},
  pages 373--463. World Sci. Publ., River Edge, NJ, 1991.

\bibitem[CCG{\etalchar{+}}20]{chalopinchepoigenevoishiraiosajda:helly}
J{\'e}r{\'e}mie Chalopin, Victor Chepoi, Anthony Genevois, Hiroshi Hirai, and
  Damian Osajda.
\newblock Helly groups.
\newblock {\em arXiv:2002.06895}, 2020.

\bibitem[CCHO20]{chalopinchepoihiraiosajda:weakly}
J\'{e}r\'{e}mie Chalopin, Victor Chepoi, Hiroshi Hirai, and Damian Osajda.
\newblock Weakly modular graphs and nonpositive curvature.
\newblock {\em Mem. Amer. Math. Soc.}, 268(1309):vi+159, 2020.

\bibitem[Cha17]{chatterji:introduction}
Indira Chatterji.
\newblock Introduction to the rapid decay property.
\newblock In {\em Around {L}anglands correspondences}, volume 691 of {\em
  Contemp. Math.}, pages 53--72. Amer. Math. Soc., Providence, RI, 2017.

\bibitem[Cla36]{clarkson:uniformly}
James~A. Clarkson.
\newblock Uniformly convex spaces.
\newblock {\em Trans. Amer. Math. Soc.}, 40(3):396--414, 1936.

\bibitem[Day44]{day:uniform}
Mahlon~M. Day.
\newblock Uniform convexity in factor and conjugate spaces.
\newblock {\em Ann. of Math. (2)}, 45:375--385, 1944.

\bibitem[Del77]{delorme:cohomologie}
Patrick Delorme.
\newblock {$1$}-cohomologie des repr\'{e}sentations unitaires des groupes de
  {L}ie semi-simples et r\'{e}solubles. {P}roduits tensoriels continus de
  repr\'{e}sentations.
\newblock {\em Bull. Soc. Math. France}, 105(3):281--336, 1977.

\bibitem[DL15]{descombeslang:convex}
Dominic Descombes and Urs Lang.
\newblock Convex geodesic bicombings and hyperbolicity.
\newblock {\em Geom. Dedicata}, 177:367--384, 2015.

\bibitem[DL16]{descombeslang:flats}
Dominic Descombes and Urs Lang.
\newblock Flats in spaces with convex geodesic bicombings.
\newblock {\em Anal. Geom. Metr. Spaces}, 4(1):68--84, 2016.

\bibitem[DMS23]{durhamminskysisto:stable}
Matthew~G. Durham, Yair~N. Minsky, and Alessandro Sisto.
\newblock Stable cubulations, bicombings, and barycenters.
\newblock {\em Geom. Topol.}, 27(6):2383--2478, 2023.

\bibitem[EW23]{engelwulff:coronas}
Alexander Engel and Christopher Wulff.
\newblock Coronas for properly combable spaces.
\newblock {\em J. Topol. Anal.}, 15(4):953--1035, 2023.

\bibitem[Gui72]{guichardet:surcohomologie:2}
Alain Guichardet.
\newblock Sur la cohomologie des groupes topologiques. {II}.
\newblock {\em Bull. Sci. Math. (2)}, 96:305--332, 1972.

\bibitem[Hae21]{haettel:lattices}
Thomas Haettel.
\newblock Lattices, injective metrics and the $k(\pi,1)$ conjecture.
\newblock {\em arXiv:2109.07891}, 2021.

\bibitem[Hae22]{haettel:link}
Thomas Haettel.
\newblock A link condition for simplicial complexes, and cub spaces.
\newblock {\em arXiv:2211.07857}, 2022.

\bibitem[Han56]{hanner:onuniform}
Olof Hanner.
\newblock On the uniform convexity of {$L^p$} and {$l^p$}.
\newblock {\em Ark. Mat.}, 3:239--244, 1956.

\bibitem[Hay21]{hayashi:polynomial}
Koyo Hayashi.
\newblock A polynomial time algorithm to compute geodesics in {$\rm CAT(0)$}
  cubical complexes.
\newblock {\em Discrete Comput. Geom.}, 65(3):636--654, 2021.

\bibitem[Her92]{herrlich:hyperconvex}
Horst Herrlich.
\newblock Hyperconvex hulls of metric spaces.
\newblock In {\em Proceedings of the {S}ymposium on {G}eneral {T}opology and
  {A}pplications ({O}xford, 1989)}, volume~44, pages 181--187, 1992.

\bibitem[Hir21]{hirai:nonpositive}
Hiroshi Hirai.
\newblock A nonpositive curvature property of modular semilattices.
\newblock {\em Geom. Dedicata}, 214:427--463, 2021.

\bibitem[HKS16]{haettelkielakschwer:strand}
Thomas Haettel, Dawid Kielak, and Petra Schwer.
\newblock The 6-strand braid group is {${\rm CAT}(0)$}.
\newblock {\em Geom. Dedicata}, 182:263--286, 2016.

\bibitem[HM11]{haissinskymathieu:conjecture}
Peter Ha{\"i}ssinsky and Pierre Mathieu.
\newblock La conjecture de {B}aum--{C}onnes pour les groupes hyperboliques par
  les marches al{\'e}atoires.
\newblock {\em Preprint available at
  \mbox{phaissin.perso.math.cnrs.fr/Doc/baumconnesgreen.pdf}}, 2011.

\bibitem[HO21]{haettelosajda:locally}
Thomas Haettel and Damian Osajda.
\newblock Locally elliptic actions, torsion groups, and nonpositively curved
  spaces.
\newblock {\em arXiv:2110.12431}, 2021.

\bibitem[Kar08]{kar:discrete}
Aditi Kar.
\newblock {\em Discrete groups and {CAT(0)} asymptotic cones}.
\newblock PhD thesis, Ohio State University, 2008.

\bibitem[Kar24]{karlsson:metric}
Anders Karlsson.
\newblock A metric fixed point theorem and some of its applications.
\newblock {\em Geom. Funct. Anal.}, 34(2):486--511, 2024.

\bibitem[KL96]{kapovichleeb:actions}
Michael Kapovich and Bernhard Leeb.
\newblock Actions of discrete groups on nonpositively curved spaces.
\newblock {\em Math. Ann.}, 306(2):341--352, 1996.

\bibitem[KL20]{kleinerlang:higher}
Bruce Kleiner and Urs Lang.
\newblock Higher rank hyperbolicity.
\newblock {\em Invent. Math.}, 221(2):597--664, 2020.

\bibitem[KS94]{kasparovskandalis:groupes}
Guennadi Kasparov and Georges Skandalis.
\newblock Groupes ``boliques'' et conjecture de {N}ovikov.
\newblock {\em C. R. Acad. Sci. Paris S\'{e}r. I Math.}, 319(8):815--820, 1994.

\bibitem[KS03]{kasparovskandalis:groups}
Gennadi Kasparov and Georges Skandalis.
\newblock Groups acting properly on ``bolic'' spaces and the {N}ovikov
  conjecture.
\newblock {\em Ann. of Math. (2)}, 158(1):165--206, 2003.

\bibitem[Laf02]{lafforgue:theorie}
Vincent Lafforgue.
\newblock {$K$}-th\'{e}orie bivariante pour les alg\`ebres de {B}anach et
  conjecture de {B}aum-{C}onnes.
\newblock {\em Invent. Math.}, 149(1):1--95, 2002.

\bibitem[Lan13]{lang:injective}
Urs Lang.
\newblock Injective hulls of certain discrete metric spaces and groups.
\newblock {\em J. Topol. Anal.}, 5(3):297--331, 2013.

\bibitem[LT79]{lindenstrausstzafriri:classical:2}
Joram Lindenstrauss and Lior Tzafriri.
\newblock {\em Classical {B}anach spaces {II}}, volume~97 of {\em Ergebnisse
  der Mathematik und ihrer Grenzgebiete}.
\newblock Springer-Verlag, Berlin-New York, 1979.

\bibitem[Mie14]{miesch:injective}
Benjamin Miesch.
\newblock Injective metrics on cube complexes.
\newblock {\em arXiv:1411.7234}, 2014.

\bibitem[Mie17]{miesch:cartan}
Benjamin Miesch.
\newblock The {C}artan-{H}adamard theorem for metric spaces with local geodesic
  bicombings.
\newblock {\em Enseign. Math.}, 63(1-2):233--247, 2017.

\bibitem[MY02]{mineyevyu:baumconnes}
Igor Mineyev and Guoliang Yu.
\newblock The {B}aum-{C}onnes conjecture for hyperbolic groups.
\newblock {\em Invent. Math.}, 149(1):97--122, 2002.

\bibitem[Oht09]{ohta:uniform}
Shin-ichi Ohta.
\newblock Uniform convexity and smoothness, and their applications in {F}insler
  geometry.
\newblock {\em Math. Ann.}, 343(3):669--699, 2009.

\bibitem[Oht21]{ohta:comparison}
Shin-ichi Ohta.
\newblock {\em Comparison {F}insler geometry}.
\newblock Springer Monographs in Mathematics. Springer, Cham, 2021.

\bibitem[OP10]{owenprovan:fast}
Megan Owen and J~Scott Provan.
\newblock A fast algorithm for computing geodesic distances in tree space.
\newblock {\em IEEE/ACM Transactions on Computational Biology and
  Bioinformatics}, 8(1):2--13, 2010.

\bibitem[PS17]{piateksamulewicz:gluing}
Bo\.{z}ena Pi\c{a}tek and Alicja Samulewicz.
\newblock Gluing {B}usemann spaces.
\newblock In {\em Selected problems on experimental mathematics}, pages
  129--148. Wydaw. Politech. \'{S}l., Gliwice, 2017.

\bibitem[Yu05]{yu:hyperbolic}
Guoliang Yu.
\newblock Hyperbolic groups admit proper affine isometric actions on
  {$l^p$}-spaces.
\newblock {\em Geom. Funct. Anal.}, 15(5):1144--1151, 2005.

\end{thebibliography}
\end{document}